\documentclass[12 pt]{article}
\usepackage{stmaryrd}
\usepackage{times}
\usepackage{booktabs}
\usepackage{subfigure}
\usepackage{adjustbox}
\usepackage{rotating}
\usepackage{diagbox}
\usepackage{tabularx}
\usepackage{multirow, makecell}
\usepackage{rotating}
\usepackage{longtable}
\usepackage{supertabular}
\usepackage{pifont}
\usepackage{floatrow}
\usepackage{caption}
\usepackage{mathrsfs}
\usepackage[fleqn]{amsmath}
\usepackage{amsfonts,amsthm,amssymb,mathrsfs,bbding}
\usepackage{txfonts}
\usepackage{graphics,multicol}
\usepackage{graphicx}
\usepackage{color}
\usepackage{caption}
\usepackage{indentfirst}
\usepackage{cite}
\usepackage{latexsym,bm}
\usepackage{enumerate}
\evensidemargin=\oddsidemargin\topmargin=-1.5cm

\newtheorem{lem}{Lemma}[section]

\newtheorem{thm}{Theorem}[section]
\newtheorem{cor}{Corollary}[section]

\theoremstyle{definition}

\newtheorem{remark}{Remark}[section]

\addtocounter{section}{0}
\begin{document}

\title{Quadratic starlike trees\footnote{Supported by National Natural Science Foundation of China (Grant Nos. 11971274, 12061074, 11671344)}}
\author{
{\small  Yarong Hu$^{1,2}$\;,\;\ \ Qiongxiang Huang$^{1,}$\footnote{Corresponding author.
\newline{\it \hspace*{5mm}Email addresses:}  huangqx@xju.edu.cn (Q. Huang).}}\\[2mm]
\footnotesize $^1$ College of Mathematics and System Science,
Xinjiang University, Urumqi 830046, China\\
\footnotesize $^2$ School of Mathematics and Information Technology, Yuncheng University, Yuncheng 044000, China }
\date{}
\maketitle {\flushleft\large\bf Abstract}
In this paper, we introduce the notion of the quadratic  graph, that is a graph whose eigenvalues are integral or quadratic algebraic integral, and  determine nine infinite families of  quadratic starlike trees, which  are just all the quadratic starlike trees  including   integral starlike trees. Thus the quadratic starlike trees are completely characterized, and  moreover, the display expressions for the
 characteristic polynomials of the quadratic starlike trees are also given.
\begin{flushleft}
\end{flushleft}
\section{Introduction}\label{se-1}
The graphs considered throughout this paper are  simple and connected. Let $G$ be a graph on $n$ vertices with \emph{adjacency matrix} $A(G)=(a_{ij})$, where $a_{ij}=1$ if the vertices $i$ and $j$ are adjacent, and $a_{ij}=0$ otherwise. The \emph{characteristic polynomial} of $G$ is defined  as $f_G(x)=det(x I-A(G))$, and the zeros of $f_G(x)$ are called the \emph{eigenvalues} of $G$, which are all real  since  $A(G)$ is real symmetric matrix, and listed as $\lambda_1(G) \geq \lambda_2(G) \geq \cdots \geq \lambda_n(G)$. The eigenvalues of $G$ together with their multiplicities are called the \emph{spectrum} of $G$, denoted by $Spec(G)$.

A number is said to be an \emph{algebraic integer} if it is a root of a  \emph{monic polynomial} ( i.e.,  a polynomial of integral coefficient with first term coefficient  $1$ ). It is clear that $f_G(x) \in\mathbb{Z}[x]$  is  a monic polynomial and so the eigenvalues of $G$ are all  real algebraic integers. Let $\alpha$ be any algebraic integer. A monic polynomial $p(x)$ is called the \emph{minimal polynomial} of $\alpha$ if $\partial(p(x))$, \emph{the degree of $p(x)$},  is as small as possible such that $p(\alpha)=0$, where $\partial(p(x))$ is also called the \emph{degree} of  $\alpha$. Obviously, any integer $a\in \mathbb{Z}$ is an algebraic integer with the minimal polynomial $p(x)=x-a$. It is well known that the minimal polynomial $p(x)$ of $\alpha$ is irreducible over rational field $\mathbb{Q}$ and is uniquely determined by $\alpha$ itself, and any other root  $\bar{\alpha}$  of $p(x)$ is called \emph{conjugate} with $\alpha$. An  algebraic integer $\alpha$ is \emph{quadratic } if its minimal polynomial $p(x)$ is of degree two.  Since $f_G(x)\in\mathbb{Z}[x]$ can be uniquely decomposed into a product of some irreducible factors, and each  irreducible factor contains some conjugate algebraic integers,  the spectrum of $G$ is a union of some conjugate algebraic integers.

A graph $G$ is called \emph{integral} if all eigenvalues of $G$  are integers, or equivalently each irreducible factor of $f_G(x)$ is of degree one. It is natural to introduce a new notion of  \emph{a quadratic graph}, that is,
a graph $G$ whose eigenvalues are the algebraic integers with degree  no more than two, or equivalently each  irreducible factor of $f_G(x)$ is of degree no more than two. It is clear that an integral graph is quadratic   but not vice versa. Moreover, a graph $G$ is called \emph{proper quadratic} if it is quadratic and there is at least one eigenvalue whose degree is exactly  two. The quadratic  graph $G$ exists, clearly complete bipartite graph $K_{n,m}$ is quadratic, and  in fact, a graph with three distinct eigenvalues is  quadratic.

The  integral graph was introduced in \cite{Harary} by Harary and Schwenk about fifty years ago. Integral graphs are very rare and difficult to find, this reduces researchers to investigate integral graphs within restricted classes of graphs, such as starlike trees, balanced trees,  some special classes of trees, trees with arbitrarily large diameters, graphs with few cycles and some Cayley graphs\cite{Watanabe, Lu1, Brouwer, Hic, Patuzzi,Omidi,Csikvari}. Also integral graphs are extended to various   spectra such as Laplacian spectrum, Signless  Laplacian spectrum \cite{Zhang, Huang}. One can refer to \cite{Lu1} for the summary of the researches of  integral graphs.

Denote by $P_i$  the path containing $i$ vertices. The starlike tree $T=T_{n_1, n_2, \ldots ,n_k}$ is a tree  on $n=n_1+2n_2+\cdots +kn_k+1$ vertices that consists of $n_i$ copies of pendant paths $P_i$, each of them shares a common vertex $u$,   i.e.,
$T-u=\cup_{i=1}^{k}n_iP_i$ where $u$ is the center vertex, and  the degree of $u$ is denoted by $d_{T}(u)$.
In this paper we focus our intention on the quadratic starlike trees. First we study the properties of eigenvalues for the quadratic starlike tree, and then give the expressions  of their characteristic polynomials, finally we completely determine all the quadratic starlike trees.

Our article is organized as follows.  In Section 2 we list and give some Lemmas for later use. In Section 3 we give two forms of characteristic polynomials for the quadratic starlike trees.  In Section 4 we identify all quadratic starlike trees which include all the integral starlike trees, and simultaneously  get their characteristic polynomials.

\section{Preliminary}\label{se-2}

\begin{lem}[Interlacing Theorem\cite{Dragos}]\label{interlace}
Let $G$ be a graph with eigenvalues $\lambda_1 \geq \lambda_2\geq\ldots \geq\lambda_n$ and $H$ be an induced subgraph of the graph $G$ with eigenvalues $\mu_1 \geq \mu_2\geq\ldots \geq\mu_m$. Then $\lambda_i\geq\mu_i\geq\lambda_{n-m+i}$ for $i=1,\ldots,m$.
\end{lem}
For an eigenvalue $\lambda$ of a graph $G$, let $m(G;\lambda)$ denote  the multiplicity of $\lambda$. It immediately follows the result from Lemma \ref{interlace}.
\begin{cor}\label{multi}
Let $G$ be a graph and $v\in V(G)$. If $\lambda$ is an eigenvalue of $G-v$, then $m(G-v;\lambda)-1\leq m(G;\lambda)\leq m(G-v;\lambda)+1$; if $\lambda$ isn't an eigenvalue of a graph $G-v$, then $m(G;\lambda)=0\  or~ 1$.
\end{cor}

If $G$ is restrict to a tree, there is a stronger result for Corollary \ref{multi} that is stated in the following lemma.
\begin{lem}[\cite{Godsil}]\label{interlace-path}
Let $T$ be a tree and $\lambda$ an eigenvalue of $T$ with multiplicity $m\geq 2$. Let $P$ be a path in $T$. Then $\lambda$ is an eigenvalue of $T-P$ with multiplicity at least $m-1$.
\end{lem}
For eigenvalue $0$ of a graph, we have the following result.
\begin{lem}[\cite{Dragos}]\label{0-mul}
Let $G$ be a graph with a pendant edge $uv$. Then $0$ has the same multiplicity as an eigenvalue of $G$ and $G-u-v$.
\end{lem}

Using Lemma \ref{0-mul}, we can give the exact value of $m(T;0)$  for   a starlike tree $T$.
\begin{cor}\label{starlike-eigen-0}
Let $T$ be a starlike tree with the center vertex $u$, and $k$ be the number of paths of even length in $T-u$.  We have $m(T;0)=k-1$ if $k\geq1$, and $m(T;0)=1$ otherwise.
\end{cor}

\begin{lem}[\cite{Dragos}]\label{lem-path-cycle-polynomial}
For the path $P_n$ and cycle $C_n$, we have  $Spec(P_n)=\{ \lambda_j=2cos\frac{\pi j}{n+1}\mid j=1,\ldots ,n\}$ and  $Spec(C_n)=\{ \lambda_j=2cos\frac{2\pi j}{n}\mid j=1,\ldots ,n \}$.
\end{lem}
From Lemma \ref{lem-path-cycle-polynomial}, we see that the largest eigenvalues $\lambda_1(P_n)=2cos\frac{\pi }{n+1}$,  $\lambda_1(C_n)=2$ and the second largest eigenvalues $\lambda_2(P_n)=2cos\frac{2\pi }{n+1}$, $\lambda_2(C_n)=2cos\frac{2\pi }{n}$.
\begin{lem}[\cite{Mirko}]\label{starlike-poly}
The characteristic polynomial of a starlike tree $T=T_{n_1, n_2, \ldots ,n_k}$ is
$f_T(x)=x\prod_{i=1}^kf^{n_i}_{P_i}(x)-\sum_{i=1}^k[n_if_{P_{i-1}}(x)f^{n_i-1}_{P_i}(x)\prod_{j\not=i}f^{n_j}_{P_j}(x)]$,
where  $f_{P_0}(x)=1$.
\end{lem}

\begin{lem}[\cite{Dragos}]\label{diam}
If $G$ is a connected graph with precisely $m$ distinct eigenvalues then the diameter $d(G)\leq m-1$.
\end{lem}

\begin{lem}[\cite{Lang}]\label{lem-Euler}
Let $n$ be an integer $>1$, and let $n=\prod_i p_i^{r_i}$ be a factorization of $n$ into primes, with exponents $r_i\geq1$. Then the Euler function
$\phi(n)=n\prod_i(1-\frac{1}{p_i})=\prod_i p_i^{r_i-1}~\prod_i(p_i-1)$.
\end{lem}

Let $\zeta_n=$exp($2\pi i/n$) denote a primitive $n$th root of unity, $\mathbb{Q}(\zeta_n+\zeta_n^{-1})$ denote the extension field  generated by $\zeta_n+\zeta_n^{-1}=2cos\frac{2\pi }{n}$ over rational field $\mathbb{Q}$,  and $[\mathbb{Q}(\zeta_n+\zeta_n^{-1}):\mathbb{Q}]$ denote the extension degree of the field $\mathbb{Q}(\zeta_n+\zeta_n^{-1})$ over rational field $\mathbb{Q}$.

\begin{lem}[\cite{Washington}]\label{lem-1}
For $n>2$, $[\mathbb{Q}(\zeta_n+\zeta_n^{-1}):\mathbb{Q}]=\frac{1}{2}\phi(n)$, where $\phi(n)$ is Euler function.
\end{lem}

From Lemma \ref{lem-1}, we see that  the algebraic degrees of $\lambda_2(P_n)$ and $ \lambda_2(C_n)$ are  $\frac{1}{2}\phi(n+1)$ and $\frac{1}{2}\phi(n)$, respectively, which can be used to identify the quadratic  pathes and cycles.

\begin{thm}\label{thm-path-cyc-1}
A path $P_n$ is quadratic  if and only if  $n\le 5$. A cycle $C_n$ is quadratic  if and only if $n\in \{3, 4,5, 6, 8, 10, 12\}$.
\end{thm}

\begin{proof}
Obviously the path $P_1$ is quadratic. Now  we consider the path $P_n$ for  $n>1$.   By Lemma \ref{lem-Euler}, we can verify that the algebraic degree
of $\lambda_2(P_n)$ satisfies:
$$ \frac{1}{2}\phi(n+1)\left\{\begin{array}{ll}
\le 2,& \mbox{  if $2\le n\le 11$ and $n\not=6,8,10$;}\\
>2, &\mbox{  if  $n=6,8,10$ or $n>11$.}
 \end{array}\right.
 $$
  It implies that all possible quadratic  paths are included in  $P_2, P_3, P_4, P_5, P_7, P_9$ and $P_{11}$. By the factorization of the characteristic polynomials of these paths, one can simply  find that $P_1, P_2, P_3, P_4, P_5$ are all the quadratic  paths.

As similar as above arguments, by considering  the algebraic degree of $\lambda_2(C_n)=2cos\frac{2\pi}{n}$,  we identify  that $C_3, C_4, C_5, C_6, C_8, C_{10}, C_{12}$ are all the  quadratic cycles.
\end{proof}
For a square-free number $N>0$, the equation $x^2-Ny^2=-1$ is known as the negative Pell equation. More generally, the infinitely many solutions of the negative Pell equation are given by the following Lemma.
\begin{lem} [\cite{Koshy}]\label{lem-pell-1}
Suppose $(x, y)=(\alpha, \beta)$ is the least solution of $x^2-Ny^2=-1$, where $N>0$ and is square-free. Then the equation has infinitely many solutions, given by $x=[(\alpha+\beta\sqrt{N})^{2k-1}+(\alpha-\beta\sqrt{N})^{2k-1}]/2$, $y=[(\alpha+\beta\sqrt{N})^{2k-1}-(\alpha-\beta\sqrt{N})^{2k-1}]/(2\sqrt{N})$ for any positive integer $k$.
\end{lem}
Particularly,  $(1, 1)$ is the least solution of $x^2-2y^2=-1$, and by  Lemma \ref{lem-pell-1}, $x=\frac{(1+\sqrt{2})^{2k-1}+(1-\sqrt{2})^{2k-1}}{2}$ and $y=\frac{(1+\sqrt{2})^{2k-1}-(1-\sqrt{2})^{2k-1}}{2\sqrt{2}}$ give infinitely many solutions for any positive integer $k$.

\begin{lem}\label{lem-pell}
Let $P(x,y)=y^2-4(x-2)$. Then $P(x,y)$ is  a square-free number for any positive integral solution $(x,y)$ of the negative Pell equation $x^2-2y^2=-1$.
\end{lem}
\begin{proof}
For any positive integral solution $(x,y)$ of  Pell equation $x^2-2y^2=-1$, we have
$$\begin{array}{ll}&\!\!\!\!P(x,y)=y^2-4(x-2)\\&\!\!\!\!=[\frac{(1+\sqrt{2})^{2k-1}-(1-\sqrt{2})^{2k-1}}{2\sqrt{2}}]^2-4(\frac{(1+\sqrt{2})^{2k-1}+(1-\sqrt{2})^{2k-1}}{2}-2)\\
&\!\!\!\!=\!\!\frac{1}{8}[((1\!+\!\sqrt{2})^{2k-1})^2\!\!+\!\!((1\!-\!\sqrt{2})^{2k-1})^2\!\!-\!\!2(1\!\!+\!\!\sqrt{2})^{2k-1}(1\!\!-\!\!\sqrt{2})^{2k-1}\!\!-\!\!16(1\!\!+\!\!\sqrt{2})^{2k-1}\!\!-\!\!16(1\!\!-\!\!\sqrt{2})^{2k-1}\!\!+\!\!64]\\
&\!\!\!\!=[\frac{(1+\sqrt{2})^{2k-1}+(1-\sqrt{2})^{2k-1}-8}{2\sqrt{2}}]^2.
\end{array}$$
Note that $\frac{(1+\sqrt{2})^{2k-1}+(1-\sqrt{2})^{2k-1}-8}{2\sqrt{2}}=\frac{2\sum_{i=0}^{k-1}\binom{2k-1}{2i}(\sqrt{2})^{2i}-8}{2\sqrt{2}}=\frac{\sum_{i=0}^{k-1}\binom{2k-1}{2i}2^{i}-4}{\sqrt{2}}$ is not integral, we claim that $P(x,y)$ is a square-free number.
\end{proof}

\section{The forms of the characteristic polynomials of quadratic starlike trees}

Notice that quadratic path is given in Theorem \ref{thm-path-cyc-1} and $K_{1,3}$ is quadratic since its characteristic polynomial  $f_{K_{1,3}}(x)=x^2(x^2-3)$, in what follows we  always assume that a starlike tree $T$ contains $K_{1,3}$ as its proper induced subgraph, i.e., $T \supset K_{1,3}$.

\begin{lem}\label{factor}
Let $p(x) \in\mathbb{Z}[x]$ be a monic irreducible  polynomial  of degree at most $2$, and its roots lie in the interval $(-2,2)$. Then  $p(x)\in \{x,\  x\pm 1,\  x^2-2,\ x^2\pm x-1,\ x^2-3\}$.
\end{lem}
\begin{proof}
First of all, if the degree of $p(x)$ equals $1$ then $p(x)=x$ or $x\pm 1$ since its roots lie in the interval $(-2,2)$. Next, suppose the degree of $p(x)$ equals $2$, that is  $p(x)=x^2-ax+b$. Let $\alpha$ and $\bar{\alpha}$ be two  roots of $p(x)$.  Since $\alpha, \bar{\alpha}\in (-2,2)$, we have  $\Delta=a^2-4b>0$,  $-4<\alpha+\overline{\alpha}=a<4$, and $p(2),p(-2)>0$, from which we deduce that \begin{equation}\label{lem-eq-1}max\{-4-2a,-4+2a\} <b< \frac{a^2}{4}<4.\end{equation}
Also note that $p(x)$ is irreducible, i.e., the discriminant $\Delta$ is a square-free number,  we can simply find that (\ref{lem-eq-1}) has only four solutions: $a=-1, b=-1$; $a=0, b=-2$; $a=0, b=-3$; $a=1, b=-1$. Therefore, $ p(x)=x^2-2$, $x^2-3$ or $x^2\pm x-1$.

The proof is completed.
\end{proof}
\begin{remark}
If $p(x) \in\mathbb{Z}[x]$ is a monic  quadratic irreducible  polynomial, and there is only one root lying in the interval $(-2,2)$, then
$p(x)$ is not necessary to be in  $\{x,\  x\pm 1,\  x^2-2,\ x^2\pm x-1,\ x^2-3\}$. For example $p(x)=x^2-5x+5$ has roots $\frac{5-\sqrt{5}}{2}\in (-2,2)$ and $\frac{5+\sqrt{5}}{2}\in (2,+\infty)$.
\end{remark}
It follows  the result from Lemma \ref{factor}.
\begin{cor}\label{factor-cor-1}
Let $p(x) \in\mathbb{Z}[x]$ be a monic irreducible  polynomial  of degree at most $2$, and its roots lie in the interval $(-2,2)$. Then the roots of $p(x)$ are included in $\{0,\pm1,\pm\sqrt{2},\pm\frac{1\pm\sqrt{5}}{2},\pm\sqrt{3}\}$.
\end{cor}
\begin{lem}\label{starlike-eig-1}
Let $T$ be a starlike tree, then $\lambda_2(T)<2$.
\end{lem}
\begin{proof}
Suppose that $T$ is a starlike tree with the center vertex $u$, then $T-u$ is the disjoint union of some paths. By Lemma \ref{interlace} and Lemma \ref{lem-path-cycle-polynomial}, we have   $\lambda_2(T) \leq \lambda_1(T-u)<2$.
\end{proof}
\begin{lem}\label{starlike-rad}
Let $T \supset K_{1,3} $ be a quadratic starlike tree. Then $\lambda_1(T)\geq2$.
\end{lem}
\begin{proof}
Suppose to the contrary that $\lambda_1(T)<2$. Then $\lambda_n(T)=-\lambda_1(T)>-2$ by the symmetry of the spectrum of a tree, and thus all the eigenvalues of $T$ are contained in the interval $(-2,2)$. Since $T$ is quadratic, from Corollary \ref{factor-cor-1} we  have $\lambda_1(T)\leq\sqrt{3}=\lambda_1(K_{1,3})$,  a contradiction.
\end{proof}

\begin{remark}
It worths mentioning that the condition of quadratic starlike tree is  necessary in Lemma \ref{starlike-rad}.   For example, let $T\not=C_n$ be  a Smith graph (Fig \ref{fig-smith}), and then $\lambda_1(T)=2$. It is easy to see that there exists some vertex $v$ of $T$ such that $T-v$ is also starlike tree containing $K_{1,3}$ as its proper subgraph,  but $\lambda_1(T-v)<2$ due to $\lambda_1(T)=2$.
\end{remark}
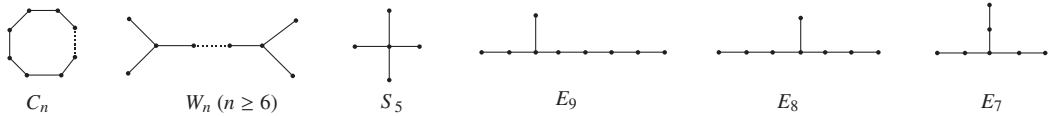
\begin{figure}[h]
\unitlength 0.15mm 
\scriptsize
\linethickness{0.4pt}
\ifx\plotpoint\undefined\newsavebox{\plotpoint}\fi 
\begin{center}
\begin{picture}(917.9,79.8)
\put(17.4,74.7){\circle*{4}}
\put(42.8,74.7){\circle*{4}}
\qbezier(17.4,74.7)(30.1,74.7)(42.8,74.7)
\put(0.0,57.3){\circle*{4}}
\qbezier(17.4,74.7)(8.7,66.0)(0.0,57.3)
\put(0.0,32.6){\circle*{4}}
\qbezier(0.0,57.3)(0.0,45.0)(0.0,32.6)
\put(15.2,17.4){\circle*{4}}
\qbezier(0.0,32.6)(7.6,25.0)(15.2,17.4)
\put(45.7,17.4){\circle*{4}}
\qbezier(15.2,17.4)(30.5,17.4)(45.7,17.4)
\put(58.0,59.5){\circle*{4}}
\qbezier(42.8,74.7)(50.4,67.1)(58.0,59.5)
\put(58.0,33.4){\circle*{4}}
\qbezier(45.7,17.4)(51.8,25.4)(58.0,33.4)
\qbezier[7](58.0,59.5)(58.0,46.4)(58.0,33.4)
\put(105.9,67.4){\circle*{4}}
\put(129.8,43.5){\circle*{4}}
\qbezier(105.9,67.4)(117.8,55.5)(129.8,43.5)
\put(105.1,18.9){\circle*{4}}
\qbezier(105.1,18.9)(117.5,31.2)(129.8,43.5)
\put(163.1,43.5){\circle*{4}}
\qbezier(129.8,43.5)(146.5,43.5)(163.1,43.5)
\put(195.0,43.5){\circle*{4}}
\qbezier[8](163.1,43.5)(179.1,43.5)(195.0,43.5)
\put(224.0,43.5){\circle*{4}}
\qbezier(195.0,43.5)(209.5,43.5)(224.0,43.5)
\put(253.8,66.0){\circle*{4}}
\qbezier(224.0,43.5)(238.9,54.7)(253.8,66.0)
\put(250.9,16.7){\circle*{4}}
\qbezier(224.0,43.5)(237.4,30.1)(250.9,16.7)
\put(306.0,42.8){\circle*{4}}
\put(336.4,42.8){\circle*{4}}
\qbezier(306.0,42.8)(321.2,42.8)(336.4,42.8)
\put(336.4,13.1){\circle*{4}}
\qbezier(336.4,42.8)(336.4,27.9)(336.4,13.1)
\put(336.4,74.7){\circle*{4}}
\qbezier(336.4,74.7)(336.4,58.7)(336.4,42.8)
\put(363.2,42.8){\circle*{4}}
\qbezier(336.4,42.8)(349.8,42.8)(363.2,42.8)
\put(418.3,37.7){\circle*{4}}
\put(443.0,37.7){\circle*{4}}
\qbezier(418.3,37.7)(430.7,37.7)(443.0,37.7)
\put(466.2,37.7){\circle*{4}}
\qbezier(443.0,37.7)(454.6,37.7)(466.2,37.7)
\put(486.5,37.7){\circle*{4}}
\qbezier(466.2,37.7)(476.3,37.7)(486.5,37.7)
\put(510.4,37.7){\circle*{4}}
\qbezier(486.5,37.7)(498.4,37.7)(510.4,37.7)
\put(533.6,37.7){\circle*{4}}
\qbezier(510.4,37.7)(522.0,37.7)(533.6,37.7)
\put(556.8,37.7){\circle*{4}}
\qbezier(533.6,37.7)(545.2,37.7)(556.8,37.7)
\put(580.7,37.7){\circle*{4}}
\qbezier(556.8,37.7)(568.8,37.7)(580.7,37.7)
\put(466.2,70.3){\circle*{4}}
\qbezier(466.2,70.3)(466.2,54.0)(466.2,37.7)
\put(628.6,37.7){\circle*{4}}
\put(653.2,37.7){\circle*{4}}
\qbezier(628.6,37.7)(640.9,37.7)(653.2,37.7)
\put(676.4,37.7){\circle*{4}}
\qbezier(653.2,37.7)(664.8,37.7)(676.4,37.7)
\put(701.1,37.7){\circle*{4}}
\qbezier(676.4,37.7)(688.8,37.7)(701.1,37.7)
\put(722.8,37.7){\circle*{4}}
\qbezier(701.1,37.7)(712.0,37.7)(722.8,37.7)
\put(746.0,37.7){\circle*{4}}
\qbezier(722.8,37.7)(734.4,37.7)(746.0,37.7)
\put(770.0,37.7){\circle*{4}}
\qbezier(746.0,37.7)(758.0,37.7)(770.0,37.7)
\put(701.1,68.2){\circle*{4}}
\qbezier(701.1,68.2)(701.1,52.9)(701.1,37.7)
\put(822.2,37.0){\circle*{4}}
\put(846.8,37.0){\circle*{4}}
\qbezier(822.2,37.0)(834.5,37.0)(846.8,37.0)
\put(868.6,37.0){\circle*{4}}
\qbezier(846.8,37.0)(857.7,37.0)(868.6,37.0)
\put(894.7,37.0){\circle*{4}}
\qbezier(868.6,37.0)(881.6,37.0)(894.7,37.0)
\put(917.9,37.0){\circle*{4}}
\qbezier(894.7,37.0)(906.3,37.0)(917.9,37.0)
\put(868.6,58.0){\circle*{4}}
\qbezier(868.6,37.0)(868.6,47.5)(868.6,58.0)
\put(868.6,79.8){\circle*{4}}
\qbezier(868.6,79.8)(868.6,68.9)(868.6,58.0)
\put(15.1,0.0){\makebox(0,0)[tl]{$C_n$}}
\put(155.7,0.0){\makebox(0,0)[tl]{$W_n~(n\geq6)$}}
\put(329.2,0.5){\makebox(0,0)[tl]{$S_5$}}
\put(483.6,2.9){\makebox(0,0)[tl]{$E_9$}}
\put(680.1,0.4){\makebox(0,0)[tl]{$E_8$}}
\put(860.6,0.6){\makebox(0,0)[tl]{$E_7$}}
\end{picture}
\end{center}
\caption{Smith graphs}\label{fig-smith}
\end{figure}
For the quadratic Smith graphs, we have the following result.  \\
(i) $C_n$ is quadratic if and only if  $n\in \{3, 4,5, 6, 8, 10, 12\}$ (see Theorem \ref{thm-path-cyc-1}).\\
(ii) $W_n$ is quadratic if and only if $6 \leq n\leq 9$.

In fact, from Section 3.3 in \cite{Brouwer2}, we know that  $Spec(W_n)=\{\pm2, 0^2, 2cos\frac{i\pi}{n-3}\mid  i=1,\ldots, n-4 \}$, and then the algebraic degree of $\lambda_3(W_n)=2cos\frac{2\pi}{n-3}$ is $\frac{1}{2}\phi(n-3)$ according to Lemma \ref{lem-1}. From Lemma \ref{lem-Euler}, we can verify that $\frac{1}{2}\phi(n-3)\leq2$ if and only if $5\leq n\leq9$ and $n=11,13,15$.  By direct calculation of   their characteristic polynomials, we know  that $W_n$  is quadratic for $6\leq n\leq9$  but $W_{11}$, $W_{13}$ and $W_{15}$ are not. \\
(iii) One can directly verify that $S_5,E_9,E_8,E_7$
are quadratic.

As an inverse proposition of Lemma \ref{starlike-rad}, we have the following result.
\begin{cor}
Let $T \supset K_{1,3} $ be a starlike tree. If $\lambda_1(T)<2$ then $T$ is not quadratic.
\end{cor}
\begin{thm}\label{star-thm-1}
Let $T \supset K_{1,3} $ be a quadratic starlike tree. Then  the characteristic polynomial $f_T(x)$ of  $T$ has the two forms:
\begin{gather}
f_T(x)=x^{t_1}(x^2-1)^{t_2}(x^2-2)^{t_3}((x^2-x-1)(x^2+x-1))^{t_4}(x^2-3)^{t_5}(x^2-c), \tag{I}
\end{gather}
where $c\geq 4$, $t_1,...,t_5\ge 0$ are all integers and $\lambda_1(T)$ is a root of $x^2-c$, or
\begin{gather}\nonumber
\!\!\!\!\!\!\!\!f_T(x)\!\!=\!\!x^{t_1}(x^2\!-\!1)^{t_2}(x^2\!-\!2)^{t_3}((x^2\!-\!x\!-\!1)(x^2\!+\!x\!-\!1))^{t_4}(x^2\!-\!3)^{t_5}
(x^2\!-\!ax\!+\!b)(x^2\!+\!ax\!+\!b), \tag{II}
\end{gather}
where $a>0$,  $t_1,...,t_5\ge 0$ are all integers;  $\Delta=a^2-4b>0$ is  a square-free number, and $\lambda_1(T)$ is a root of $x^2-ax+b$.
\end{thm}
\begin{proof}
Let $\lambda$ be any eigenvalue of $T$. Then $\lambda$ is a root of some irreducible factor $p(x)$ of $f_T(x)$. Since $T$ is quadratic, the degree of $p(x)$ is at most two.

First suppose that $|\lambda|=\lambda_1(T)$. If $\partial(p(x))=1$ then $p(x)=(x-\lambda)$. Also $f_T(x)$ has a factor $(x+\lambda)$ because  $T$ is bipartite, and  in this case, $f_T(x)$ is of the form (I). If $\partial(p(x))=2$, we may assume that $p(x)=x^2-ax+b$ contains $\lambda_1(T)=\frac{a+\sqrt{a^2-4b}}{2}$ as its root, where $a\ge 0$  and $\Delta=a^2-4b>0$ is a square-free number due to $p(x)$ is irreducible. If $a=0$ then $p(x)=x^2+b=x^2-c$, where $c=-b>0$, and so  $f_T(x)$ is of the form (I). Otherwise, $\frac{a\pm \sqrt{a^2-4b}}{2}$ are two roots of $p(x)$,  where $a>0$. Notice that $-\frac{a\pm \sqrt{a^2-4b}}{2}$ are also eigenvalues of $T$, and they are obviously the two roots of irreducible   $x^2+ax+b$. Since $-\frac{a+ \sqrt{a^2-4b}}{2}$ and the spectral radius $\frac{a+ \sqrt{a^2-4b}}{2}$ are simple, we claim that the irreducible factors $x^2-ax+b, x^2+ax+b|f_T(x)$ and their multiplicities in $f_T(x)$ are  the same and equal  to $1$. Thus $f_T(x)$ is of the form (II) in this case.

Next assume that two roots  $|\lambda|,|\bar{\lambda}|<\lambda_1(T)$, where $\bar{\lambda}$ is the conjugate of $\lambda$ and $\lambda=\bar{\lambda}$ if $\partial(p(x))=1$. Then $|\lambda|,|\bar{\lambda}|\le\lambda_2(T)<2$ by Lemma \ref{starlike-eig-1}. If $\partial(p(x))=1$ then $p(x)\in\{x,x\pm 1\}$ by Lemma \ref{factor}. If $\partial(p(x))=2$ then  we have $p(x)\in \{x^2-2,\ x^2-3,\ x^2\pm x-1\}$ by Lemma \ref{factor}. Additionally, we mention that if $p(x)=x^2+ x-1$ then $f_T(x)$ also contains $x^2-x-1$ because their roots happen to be the opposite number, and vice versa. It implies that  $(x^2+x-1)^t\parallel f_T(x)$ if and only if  $(x^2-x-1)^t\parallel f_T(x)$ (i.e., $(x^2-x-1)^t \mid f_T(x)$, but $(x^2-x-1)^{t+1} \nmid f_T(x)$ ).

We complete this proof.
\end{proof}

\begin{remark}
In the express of $f_T(x)$ in Theorem \ref{star-thm-1}, we mention that the factor will disappear if the corresponding power $t_i=0$. Note that each factor of  $x$, $x^2-1$, $x^2-2$, $(x^2-x-1)(x^2+x-1)$ and $x^2-3$  contains the opposite numbers as their roots, i.e., the roots of each factor are opposite closed, we call them the \emph{basis factors} of $f_T(x)$ and collect them in the set $S=\{x, x^2-1, x^2-2, (x^2-x+1)(x^2+x+1), x^2-3\}$. In fact, the product of the irreducible factors of  $\prod_{i=1}^{5} f_{P_i}(x)$ is just $\prod_{p(x)\in S}p(x)= x(x^2-1)(x^2-2)(x^2-x-1)(x^2+x-1)(x^2-3)$.
\end{remark}

It immediately follows the following result from Theorem \ref{star-thm-1}.

\begin{cor}\label{starlike-the-third-largest-1}
Let $T \supset K_{1,3} $ be a quadratic starlike tree and $\Lambda(T)$ be the set of possible  eigenvalues of $T$. We have
$$\Lambda(T) = \left\{\begin{array}{ll}
\{0,\pm1,\pm\sqrt{2}, \pm\frac{1\pm\sqrt{5}}{2},\pm \sqrt{3}, \pm\sqrt{c}\},& \mbox{ if $f_T(x)$ is of the form (I);}\\
\{0,\pm1,\pm\sqrt{2}, \pm\frac{1\pm\sqrt{5}}{2},\pm \sqrt{3},\pm\frac{a\pm \sqrt{a^2-4b}}{2}\},& \mbox{ if $f_T(x)$ is of the form (II),}
\end{array}\right.
$$
where integers $c\geq4$, $a>0$ and $a^2-4b>0$ is  a square-free number.
\end{cor}
\begin{cor}\label{starlike-the-third-largest}
Let $T \supset K_{1,3} $ be a quadratic starlike tree. Then \\
(i) if $f_T(x)$ is  of the form (I) then  $\lambda_2(T)\leq \sqrt{3}$; if $f_T(x)$ is  of the form (II) then $\lambda_3(T)\leq \sqrt{3}$.\\
(ii) $d(T)\leq 14$.
\end{cor}
\begin{proof}
(i) immediately  follows  from the forms (\uppercase\expandafter{\romannumeral1}) and (\uppercase\expandafter{\romannumeral2}) of $f_T(T)$.

From  Corollary \ref{starlike-the-third-largest-1}, we see that  the number of distinct eigenvalues of $T$ is at most $15$. Hence  $d(T)\leq 14$ by Lemma \ref{diam} and (ii) follows.
\end{proof}

\section{Characterization of quadratic starlike trees}

\begin{lem}\label{starlike-two-path}
Let $T \supset K_{1,3} $ be a quadratic starlike tree with the center vertex $u$. Then $T-u$  does not contain two  copies of path $P_i$ such that  $i\ge6$.
\end{lem}
\begin{proof}
Suppose to the contrary that $T-u$  contains two $P_i$ where $i\ge6$. By Theorem \ref{thm-path-cyc-1},  $P_i$ has an  eigenvalue $\lambda$ whose  degree  is greater than $2$. Additionally, the multiplicity $m(T-u;\lambda)\ge2$. Thus $\lambda$ is an eigenvalue of $T$ by Lemma \ref{interlace}. However $T$ is quadratic, this is impossible.
\end{proof}

\begin{lem}\label{starlike-two-path-distinct}
Let $T \supset K_{1,3} $ be a quadratic starlike tree with the center vertex $u$. Then $T-u$  can not contain two  paths $P_i$ and $P_j$ such that $i>j\ge 6$.
\end{lem}
\begin{proof}
Suppose to the contrary that $T-u$  contains $P_i$ and $P_j$ for some $i>j \geq 6$. By Corollary \ref{starlike-the-third-largest}(2), we have $d(T)\leq 14$. It implies that  $i=7, j=6$ or $i=8, j=6$.

First we assume that  $T-u$  contains $P_7$ and $P_6$. Let $H_i$ be the starlike tree described in Fig. \ref{fig-case1-two-path} for $i=1,2,...,11$.  By direct calculation, we find that $\lambda_3(H_1),\lambda_3(H_2),\lambda_3(H_3)>\sqrt{3}$. Thus, by Lemma \ref{interlace} and Corollary \ref{starlike-the-third-largest}(1), any one of $H_1, H_2$ and $H_3$ can not become the induced subgraph of   $T$. Moreover, it deduces that $3\le d_T(u)\le 4$ because $H_3\not\subset T$. Thus,   if $d_T(u)=3$ then $T=H_4, H_5$ or $H_6$ since $H_1\not\subset T$; if $d_T(u)=4$ then $T=H_7$ since $H_2\not\subset T$. However, by direct calculation, we know that  any one of $H_4, H_5, H_6$ and $H_7$ is not quadratic. This is a contradiction.

Next we assume that $T-u$  contains $P_8$ and $P_6$.  As similar as above, $T-u$ can not contain $H_8$ and $H_9$ as its subgraph since,   by direct calculation, we have $\lambda_3(H_8),\lambda_3(H_9)>\sqrt{3}$.  Since $H_{9}\not\subset T$, we have $d_T(u)=3$. It implies that  $T=H_{10}$ or $ H_{11}$ because of  $H_{8}\not\subset T$.  However, one can directly verify that  both $H_{10}$ and $ H_{11}$ are not quadratic. It leads to a contradiction.

We complete this proof.
\end{proof}
\begin{figure}[h]
\unitlength 0.13mm 
\scriptsize
\linethickness{0.4pt}
\ifx\plotpoint\undefined\newsavebox{\plotpoint}\fi 
\begin{center}
\begin{picture}(1047.6,165.3)
\put(28.3,162.4){\circle*{4}}
\put(0.0,134.1){\circle*{4}}
\qbezier(28.3,162.4)(14.1,148.3)(0.0,134.1)
\put(0.0,114.6){\circle*{4}}
\qbezier(0.0,134.1)(0.0,124.3)(0.0,114.6)
\put(0.0,94.3){\circle*{4}}
\qbezier(0.0,114.6)(0.0,104.4)(0.0,94.3)
\put(0.0,76.9){\circle*{4}}
\qbezier(0.0,94.3)(0.0,85.6)(0.0,76.9)
\put(28.3,134.9){\circle*{4}}
\qbezier(28.3,162.4)(28.3,148.6)(28.3,134.9)
\put(28.3,115.3){\circle*{4}}
\qbezier(28.3,134.9)(28.3,125.1)(28.3,115.3)
\put(28.3,95.0){\circle*{4}}
\qbezier(28.3,115.3)(28.3,105.1)(28.3,95.0)
\put(28.3,78.3){\circle*{4}}
\qbezier(28.3,95.0)(28.3,86.6)(28.3,78.3)
\put(28.3,60.2){\circle*{4}}
\qbezier(28.3,78.3)(28.3,69.2)(28.3,60.2)
\put(28.3,39.2){\circle*{4}}
\qbezier(28.3,60.2)(28.3,49.7)(28.3,39.2)
\put(55.8,134.9){\circle*{4}}
\qbezier(28.3,162.4)(42.1,148.6)(55.8,134.9)
\put(55.8,114.6){\circle*{4}}
\qbezier(55.8,134.9)(55.8,124.7)(55.8,114.6)
\put(55.8,94.3){\circle*{4}}
\qbezier(55.8,114.6)(55.8,104.4)(55.8,94.3)
\put(55.8,77.6){\circle*{4}}
\qbezier(55.8,94.3)(55.8,85.9)(55.8,77.6)
\put(55.8,58.7){\circle*{4}}
\qbezier(55.8,77.6)(55.8,68.2)(55.8,58.7)
\put(55.8,39.2){\circle*{4}}
\qbezier(55.8,58.7)(55.8,48.9)(55.8,39.2)
\put(55.8,21.0){\circle*{4}}
\qbezier(55.8,39.2)(55.8,30.1)(55.8,21.0)
\put(124.7,161.7){\circle*{4}}
\put(94.3,131.2){\circle*{4}}
\qbezier(124.7,161.7)(109.5,146.5)(94.3,131.2)
\put(109.5,132.0){\circle*{4}}
\qbezier(124.7,161.7)(117.1,146.8)(109.5,132.0)
\put(134.1,132.7){\circle*{4}}
\qbezier(124.7,161.7)(129.4,147.2)(134.1,132.7)
\put(155.2,131.2){\circle*{4}}
\qbezier(124.7,161.7)(139.9,146.5)(155.2,131.2)
\put(109.5,111.7){\circle*{4}}
\qbezier(109.5,132.0)(109.5,121.8)(109.5,111.7)
\put(134.1,112.4){\circle*{4}}
\qbezier(134.1,132.7)(134.1,122.5)(134.1,112.4)
\put(134.1,94.3){\circle*{4}}
\qbezier(134.1,112.4)(134.1,103.3)(134.1,94.3)
\put(134.1,74.7){\circle*{4}}
\qbezier(134.1,94.3)(134.1,84.5)(134.1,74.7)
\put(134.1,57.3){\circle*{4}}
\qbezier(134.1,74.7)(134.1,66.0)(134.1,57.3)
\put(134.1,37.7){\circle*{4}}
\qbezier(134.1,57.3)(134.1,47.5)(134.1,37.7)
\put(155.2,112.4){\circle*{4}}
\qbezier(155.2,131.2)(155.2,121.8)(155.2,112.4)
\put(155.2,94.3){\circle*{4}}
\qbezier(155.2,112.4)(155.2,103.3)(155.2,94.3)
\put(155.2,74.7){\circle*{4}}
\qbezier(155.2,94.3)(155.2,84.5)(155.2,74.7)
\put(155.2,57.3){\circle*{4}}
\qbezier(155.2,74.7)(155.2,66.0)(155.2,57.3)
\put(155.2,38.4){\circle*{4}}
\qbezier(155.2,57.3)(155.2,47.9)(155.2,38.4)
\put(155.2,18.9){\circle*{4}}
\qbezier(155.2,38.4)(155.2,28.6)(155.2,18.9)
\put(236.4,163.9){\circle*{4}}
\put(195.8,134.1){\circle*{4}}
\qbezier(236.4,163.9)(216.1,149.0)(195.8,134.1)
\put(212.4,133.4){\circle*{4}}
\qbezier(236.4,163.9)(224.4,148.6)(212.4,133.4)
\put(226.2,133.4){\circle*{4}}
\qbezier(236.4,163.9)(231.3,148.6)(226.2,133.4)
\put(245.8,133.4){\circle*{4}}
\qbezier(236.4,163.9)(241.1,148.6)(245.8,133.4)
\put(245.8,115.3){\circle*{4}}
\qbezier(245.8,133.4)(245.8,124.3)(245.8,115.3)
\put(245.8,96.4){\circle*{4}}
\qbezier(245.8,115.3)(245.8,105.9)(245.8,96.4)
\put(245.8,76.1){\circle*{4}}
\qbezier(245.8,96.4)(245.8,86.3)(245.8,76.1)
\put(245.8,57.3){\circle*{4}}
\qbezier(245.8,76.1)(245.8,66.7)(245.8,57.3)
\put(245.8,39.2){\circle*{4}}
\qbezier(245.8,57.3)(245.8,48.2)(245.8,39.2)
\put(266.1,134.1){\circle*{4}}
\qbezier(236.4,163.9)(251.2,149.0)(266.1,134.1)
\put(266.1,114.6){\circle*{4}}
\qbezier(266.1,134.1)(266.1,124.3)(266.1,114.6)
\put(266.1,96.4){\circle*{4}}
\qbezier(266.1,114.6)(266.1,105.5)(266.1,96.4)
\put(266.1,76.1){\circle*{4}}
\qbezier(266.1,96.4)(266.1,86.3)(266.1,76.1)
\put(266.1,56.6){\circle*{4}}
\qbezier(266.1,76.1)(266.1,66.3)(266.1,56.6)
\put(266.1,39.2){\circle*{4}}
\qbezier(266.1,56.6)(266.1,47.9)(266.1,39.2)
\put(266.1,19.6){\circle*{4}}
\qbezier(266.1,39.2)(266.1,29.4)(266.1,19.6)
\put(338.6,163.1){\circle*{4}}
\put(308.9,133.4){\circle*{4}}
\qbezier(338.6,163.1)(323.7,148.3)(308.9,133.4)
\put(338.6,134.1){\circle*{4}}
\qbezier(338.6,163.1)(338.6,148.6)(338.6,134.1)
\put(338.6,114.6){\circle*{4}}
\qbezier(338.6,134.1)(338.6,124.3)(338.6,114.6)
\put(338.6,95.7){\circle*{4}}
\qbezier(338.6,114.6)(338.6,105.1)(338.6,95.7)
\put(338.6,76.1){\circle*{4}}
\qbezier(338.6,95.7)(338.6,85.9)(338.6,76.1)
\put(338.6,58.0){\circle*{4}}
\qbezier(338.6,76.1)(338.6,67.1)(338.6,58.0)
\put(338.6,37.7){\circle*{4}}
\qbezier(338.6,58.0)(338.6,47.9)(338.6,37.7)
\put(367.6,134.1){\circle*{4}}
\qbezier(338.6,163.1)(353.1,148.6)(367.6,134.1)
\put(367.6,114.6){\circle*{4}}
\qbezier(367.6,134.1)(367.6,124.3)(367.6,114.6)
\put(367.6,95.7){\circle*{4}}
\qbezier(367.6,114.6)(367.6,105.1)(367.6,95.7)
\put(367.6,76.9){\circle*{4}}
\qbezier(367.6,95.7)(367.6,86.3)(367.6,76.9)
\put(367.6,58.0){\circle*{4}}
\qbezier(367.6,76.9)(367.6,67.4)(367.6,58.0)
\put(367.6,37.7){\circle*{4}}
\qbezier(367.6,58.0)(367.6,47.9)(367.6,37.7)
\put(367.6,20.3){\circle*{4}}
\qbezier(367.6,37.7)(367.6,29.0)(367.6,20.3)
\put(445.9,163.9){\circle*{4}}
\put(416.9,134.9){\circle*{4}}
\qbezier(445.9,163.9)(431.4,149.4)(416.9,134.9)
\put(445.9,134.9){\circle*{4}}
\qbezier(445.9,163.9)(445.9,149.4)(445.9,134.9)
\put(473.4,135.6){\circle*{4}}
\qbezier(445.9,163.9)(459.7,149.7)(473.4,135.6)
\put(445.9,116.0){\circle*{4}}
\qbezier(445.9,134.9)(445.9,125.4)(445.9,116.0)
\put(445.9,95.7){\circle*{4}}
\qbezier(445.9,116.0)(445.9,105.9)(445.9,95.7)
\put(445.9,76.9){\circle*{4}}
\qbezier(445.9,95.7)(445.9,86.3)(445.9,76.9)
\put(445.9,58.0){\circle*{4}}
\qbezier(445.9,76.9)(445.9,67.4)(445.9,58.0)
\put(445.9,37.0){\circle*{4}}
\qbezier(445.9,58.0)(445.9,47.5)(445.9,37.0)
\put(473.4,116.0){\circle*{4}}
\qbezier(473.4,135.6)(473.4,125.8)(473.4,116.0)
\put(473.4,95.7){\circle*{4}}
\qbezier(473.4,116.0)(473.4,105.9)(473.4,95.7)
\put(473.4,76.9){\circle*{4}}
\qbezier(473.4,95.7)(473.4,86.3)(473.4,76.9)
\put(473.4,58.0){\circle*{4}}
\qbezier(473.4,76.9)(473.4,67.4)(473.4,58.0)
\put(473.4,37.7){\circle*{4}}
\qbezier(473.4,58.0)(473.4,47.9)(473.4,37.7)
\put(473.4,19.6){\circle*{4}}
\qbezier(473.4,37.7)(473.4,28.6)(473.4,19.6)
\put(416.9,118.2){\circle*{4}}
\qbezier(416.9,134.9)(416.9,126.5)(416.9,118.2)
\put(540.9,163.9){\circle*{4}}
\put(511.1,134.1){\circle*{4}}
\qbezier(540.9,163.9)(526.0,149.0)(511.1,134.1)
\put(570.6,134.1){\circle*{4}}
\qbezier(540.9,163.9)(555.7,149.0)(570.6,134.1)
\put(540.9,134.9){\circle*{4}}
\qbezier(540.9,163.9)(540.9,149.4)(540.9,134.9)
\put(540.9,116.0){\circle*{4}}
\qbezier(540.9,134.9)(540.9,125.4)(540.9,116.0)
\put(540.9,95.0){\circle*{4}}
\qbezier(540.9,116.0)(540.9,105.5)(540.9,95.0)
\put(540.9,76.1){\circle*{4}}
\qbezier(540.9,95.0)(540.9,85.6)(540.9,76.1)
\put(540.9,57.3){\circle*{4}}
\qbezier(540.9,76.1)(540.9,66.7)(540.9,57.3)
\put(540.9,37.0){\circle*{4}}
\qbezier(540.9,57.3)(540.9,47.1)(540.9,37.0)
\put(570.6,116.0){\circle*{4}}
\qbezier(570.6,134.1)(570.6,125.1)(570.6,116.0)
\put(570.6,95.0){\circle*{4}}
\qbezier(570.6,116.0)(570.6,105.5)(570.6,95.0)
\put(570.6,75.4){\circle*{4}}
\qbezier(570.6,95.0)(570.6,85.2)(570.6,75.4)
\put(570.6,56.6){\circle*{4}}
\qbezier(570.6,75.4)(570.6,66.0)(570.6,56.6)
\put(570.6,36.3){\circle*{4}}
\qbezier(570.6,56.6)(570.6,46.4)(570.6,36.3)
\put(570.6,18.9){\circle*{4}}
\qbezier(570.6,36.3)(570.6,27.6)(570.6,18.9)
\put(511.1,115.3){\circle*{4}}
\qbezier(511.1,134.1)(511.1,124.7)(511.1,115.3)
\put(511.1,94.3){\circle*{4}}
\qbezier(511.1,115.3)(511.1,104.8)(511.1,94.3)
\put(634.4,163.9){\circle*{4}}
\put(605.4,134.9){\circle*{4}}
\qbezier(634.4,163.9)(619.9,149.4)(605.4,134.9)
\put(621.3,134.1){\circle*{4}}
\qbezier(634.4,163.9)(627.9,149.0)(621.3,134.1)
\put(646.0,133.4){\circle*{4}}
\qbezier(634.4,163.9)(640.2,148.6)(646.0,133.4)
\put(663.4,134.9){\circle*{4}}
\qbezier(634.4,163.9)(648.9,149.4)(663.4,134.9)
\put(646.0,113.8){\circle*{4}}
\qbezier(646.0,133.4)(646.0,123.6)(646.0,113.8)
\put(646.0,94.3){\circle*{4}}
\qbezier(646.0,113.8)(646.0,104.0)(646.0,94.3)
\put(646.0,76.1){\circle*{4}}
\qbezier(646.0,94.3)(646.0,85.2)(646.0,76.1)
\put(646.0,55.8){\circle*{4}}
\qbezier(646.0,76.1)(646.0,66.0)(646.0,55.8)
\put(646.0,36.3){\circle*{4}}
\qbezier(646.0,55.8)(646.0,46.0)(646.0,36.3)
\put(663.4,113.8){\circle*{4}}
\qbezier(663.4,134.9)(663.4,124.3)(663.4,113.8)
\put(663.4,94.3){\circle*{4}}
\qbezier(663.4,113.8)(663.4,104.0)(663.4,94.3)
\put(663.4,76.1){\circle*{4}}
\qbezier(663.4,94.3)(663.4,85.2)(663.4,76.1)
\put(663.4,55.8){\circle*{4}}
\qbezier(663.4,76.1)(663.4,66.0)(663.4,55.8)
\put(663.4,36.3){\circle*{4}}
\qbezier(663.4,55.8)(663.4,46.0)(663.4,36.3)
\put(663.4,18.1){\circle*{4}}
\qbezier(663.4,36.3)(663.4,27.2)(663.4,18.1)
\put(742.4,164.6){\circle*{4}}
\put(721.4,134.9){\circle*{4}}
\qbezier(742.4,164.6)(731.9,149.7)(721.4,134.9)
\put(742.4,134.9){\circle*{4}}
\qbezier(742.4,164.6)(742.4,149.7)(742.4,134.9)
\put(767.1,134.1){\circle*{4}}
\qbezier(742.4,164.6)(754.7,149.4)(767.1,134.1)
\put(721.4,115.3){\circle*{4}}
\qbezier(721.4,134.9)(721.4,125.1)(721.4,115.3)
\put(721.4,93.5){\circle*{4}}
\qbezier(721.4,115.3)(721.4,104.4)(721.4,93.5)
\put(742.4,115.3){\circle*{4}}
\qbezier(742.4,134.9)(742.4,125.1)(742.4,115.3)
\put(742.4,94.3){\circle*{4}}
\qbezier(742.4,115.3)(742.4,104.8)(742.4,94.3)
\put(742.4,76.9){\circle*{4}}
\qbezier(742.4,94.3)(742.4,85.6)(742.4,76.9)
\put(742.4,55.8){\circle*{4}}
\qbezier(742.4,76.9)(742.4,66.3)(742.4,55.8)
\put(742.4,37.0){\circle*{4}}
\qbezier(742.4,55.8)(742.4,46.4)(742.4,37.0)
\put(767.1,115.3){\circle*{4}}
\qbezier(767.1,134.1)(767.1,124.7)(767.1,115.3)
\put(767.1,95.0){\circle*{4}}
\qbezier(767.1,115.3)(767.1,105.1)(767.1,95.0)
\put(767.1,77.6){\circle*{4}}
\qbezier(767.1,95.0)(767.1,86.3)(767.1,77.6)
\put(767.1,55.8){\circle*{4}}
\qbezier(767.1,77.6)(767.1,66.7)(767.1,55.8)
\put(767.1,36.3){\circle*{4}}
\qbezier(767.1,55.8)(767.1,46.0)(767.1,36.3)
\put(767.1,18.9){\circle*{4}}
\qbezier(767.1,36.3)(767.1,27.6)(767.1,18.9)
\put(837.4,165.3){\circle*{4}}
\put(806.9,134.9){\circle*{4}}
\qbezier(837.4,165.3)(822.2,150.1)(806.9,134.9)
\put(821.4,134.9){\circle*{4}}
\qbezier(837.4,165.3)(829.4,150.1)(821.4,134.9)
\put(843.9,134.9){\circle*{4}}
\qbezier(837.4,165.3)(840.6,150.1)(843.9,134.9)
\put(861.3,134.9){\circle*{4}}
\qbezier(837.4,165.3)(849.3,150.1)(861.3,134.9)
\put(843.9,114.6){\circle*{4}}
\qbezier(843.9,134.9)(843.9,124.7)(843.9,114.6)
\put(843.9,96.4){\circle*{4}}
\qbezier(843.9,114.6)(843.9,105.5)(843.9,96.4)
\put(843.9,77.6){\circle*{4}}
\qbezier(843.9,96.4)(843.9,87.0)(843.9,77.6)
\put(843.9,60.2){\circle*{4}}
\qbezier(843.9,77.6)(843.9,68.9)(843.9,60.2)
\put(843.9,38.4){\circle*{4}}
\qbezier(843.9,60.2)(843.9,49.3)(843.9,38.4)
\put(861.3,114.6){\circle*{4}}
\qbezier(861.3,134.9)(861.3,124.7)(861.3,114.6)
\put(861.3,96.4){\circle*{4}}
\qbezier(861.3,114.6)(861.3,105.5)(861.3,96.4)
\put(861.3,77.6){\circle*{4}}
\qbezier(861.3,96.4)(861.3,87.0)(861.3,77.6)
\put(861.3,59.5){\circle*{4}}
\qbezier(861.3,77.6)(861.3,68.5)(861.3,59.5)
\put(861.3,38.4){\circle*{4}}
\qbezier(861.3,59.5)(861.3,48.9)(861.3,38.4)
\put(861.3,19.6){\circle*{4}}
\qbezier(861.3,38.4)(861.3,29.0)(861.3,19.6)
\put(767.1,1.5){\circle*{4}}
\qbezier(767.1,18.9)(767.1,10.2)(767.1,1.5)
\put(861.3,1.5){\circle*{4}}
\qbezier(861.3,19.6)(861.3,10.5)(861.3,1.5)
\put(924.4,164.6){\circle*{4}}
\put(902.6,134.9){\circle*{4}}
\qbezier(924.4,164.6)(913.5,149.7)(902.6,134.9)
\put(924.4,134.9){\circle*{4}}
\qbezier(924.4,164.6)(924.4,149.7)(924.4,134.9)
\put(924.4,115.3){\circle*{4}}
\qbezier(924.4,134.9)(924.4,125.1)(924.4,115.3)
\put(924.4,97.9){\circle*{4}}
\qbezier(924.4,115.3)(924.4,106.6)(924.4,97.9)
\put(924.4,78.3){\circle*{4}}
\qbezier(924.4,97.9)(924.4,88.1)(924.4,78.3)
\put(924.4,60.2){\circle*{4}}
\qbezier(924.4,78.3)(924.4,69.2)(924.4,60.2)
\put(924.4,38.4){\circle*{4}}
\qbezier(924.4,60.2)(924.4,49.3)(924.4,38.4)
\put(954.1,134.9){\circle*{4}}
\qbezier(924.4,164.6)(939.2,149.7)(954.1,134.9)
\put(954.1,113.8){\circle*{4}}
\qbezier(954.1,134.9)(954.1,124.3)(954.1,113.8)
\put(954.1,97.2){\circle*{4}}
\qbezier(954.1,113.8)(954.1,105.5)(954.1,97.2)
\qbezier(954.1,97.2)(953.7,90.6)(953.4,84.1)
\put(954.1,78.3){\circle*{4}}
\qbezier(954.1,97.2)(954.1,87.7)(954.1,78.3)
\put(954.1,59.5){\circle*{4}}
\qbezier(954.1,78.3)(954.1,68.9)(954.1,59.5)
\put(954.1,38.4){\circle*{4}}
\qbezier(954.1,59.5)(954.1,48.9)(954.1,38.4)
\put(954.1,19.6){\circle*{4}}
\qbezier(954.1,38.4)(954.1,29.0)(954.1,19.6)
\put(954.1,1.5){\circle*{4}}
\qbezier(954.1,19.6)(954.1,10.5)(954.1,1.5)
\put(1025.2,163.1){\circle*{4}}
\put(1002.7,132.0){\circle*{4}}
\qbezier(1025.2,163.1)(1013.9,147.5)(1002.7,132.0)
\put(1025.2,132.0){\circle*{4}}
\qbezier(1025.2,163.1)(1025.2,147.5)(1025.2,132.0)
\put(1025.2,113.8){\circle*{4}}
\qbezier(1025.2,132.0)(1025.2,122.9)(1025.2,113.8)
\put(1025.2,95.7){\circle*{4}}
\qbezier(1025.2,113.8)(1025.2,104.8)(1025.2,95.7)
\put(1025.2,76.1){\circle*{4}}
\qbezier(1025.2,95.7)(1025.2,85.9)(1025.2,76.1)
\put(1025.2,56.6){\circle*{4}}
\qbezier(1025.2,76.1)(1025.2,66.3)(1025.2,56.6)
\put(1025.2,35.5){\circle*{4}}
\qbezier(1025.2,56.6)(1025.2,46.0)(1025.2,35.5)
\put(1047.6,132.7){\circle*{4}}
\qbezier(1025.2,163.1)(1036.4,147.9)(1047.6,132.7)
\put(1047.6,113.8){\circle*{4}}
\qbezier(1047.6,132.7)(1047.6,123.3)(1047.6,113.8)
\put(1047.6,95.7){\circle*{4}}
\qbezier(1047.6,113.8)(1047.6,104.8)(1047.6,95.7)
\put(1047.6,76.1){\circle*{4}}
\qbezier(1047.6,95.7)(1047.6,85.9)(1047.6,76.1)
\put(1047.6,56.6){\circle*{4}}
\qbezier(1047.6,76.1)(1047.6,66.3)(1047.6,56.6)
\put(1047.6,35.5){\circle*{4}}
\qbezier(1047.6,56.6)(1047.6,46.0)(1047.6,35.5)
\put(1047.6,17.4){\circle*{4}}
\qbezier(1047.6,35.5)(1047.6,26.5)(1047.6,17.4)
\put(1047.6,0.0){\circle*{4}}
\qbezier(1047.6,17.4)(1047.6,8.7)(1047.6,0.0)
\put(1002.7,113.8){\circle*{4}}
\qbezier(1002.7,132.0)(1002.7,122.9)(1002.7,113.8)
\put(12.3,5.1){\makebox(0,0)[tl]{$H_1$}}
\put(121.8,5.1){\makebox(0,0)[tl]{$H_2$}}
\put(222.6,5.1){\makebox(0,0)[tl]{$H_3$}}
\put(329.9,5.8){\makebox(0,0)[tl]{$H_4$}}
\put(440.1,5.8){\makebox(0,0)[tl]{$H_5$}}
\put(531.4,5.8){\makebox(0,0)[tl]{$H_6$}}
\put(630.0,5.1){\makebox(0,0)[tl]{$H_7$}}
\put(730.8,5.1){\makebox(0,0)[tl]{$H_8$}}
\put(822.9,5.1){\makebox(0,0)[tl]{$H_9$}}
\put(912.8,5.1){\makebox(0,0)[tl]{$H_{10}$}}
\put(1012.8,5.8){\makebox(0,0)[tl]{$H_{11}$}}
\end{picture}
\end{center}
  \caption{Some starlike trees}\label{fig-case1-two-path}
\end{figure}
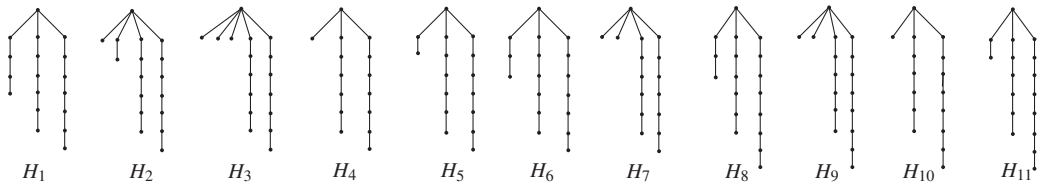

\begin{lem}\label{starlike-mul}
Let $T \supset K_{1,3} $ be a  starlike tree with the center vertex $u$,  and $\lambda$ be an eigenvalue of $T-u$. Then $m(T;\lambda)= m(T-u;\lambda)-1$.
\end{lem}
\begin{proof}
By assumption,  $\lambda$ is an eigenvalue of $T-u$ with multiplicity $m=m(T-u;\lambda)\ge 1$.  It is well known that  each  eigenvalue of a path is simple, and  $T-u$ is a union of some disjoint paths,  there are exactly $m$ paths, say $P_{k_1},P_{k_2},\ldots,P_{k_m}$ in $T-u$, such that $f_{P_{k_i}}(\lambda)=0$ for $i=1,2,\ldots,m$.  We know $m-1\leq m(T;\lambda) \leq m+1$ by Corollary \ref{multi}. Suppose to the contrary that $m(T;\lambda)\neq m-1$. Then $m(T;\lambda)=m$ or $m+1$.

Assume that $m \geq 2$. Let $P$ be the path  obtained by attaching   one end points of the paths $P_{k_1},P_{k_2}$ to the center vertex $u$, i.e., $P=P_{k_1}+u+P_{k_2}$. By Lemma \ref{interlace-path}, $\lambda$ is an eigenvalue of $T-P$ with  multiplicity at least $m(T;\lambda)-1$. However, $T-P$ is a union of some pathes in which only  $P_{k_3},...,P_{k_m}$ have $\lambda$ as its simple eigenvalues, it means that $m(T-P;\lambda)=m(T;\lambda)-2$. This is a contradiction.

Assume that $m =1$. Then, according to our assumption, there is only one path $P_{k_1}$ in $T-u$ that contains $\lambda$ as its eigenvalue. On the other hand, by applying Lemma \ref{starlike-poly} on $T$, we have $f_T(\lambda)\neq 0$, which contradicts that $m(T;\lambda)=m=1$.

The proof is completed.
\end{proof}

Denote by $m(x)=[f_{P_1}(x),f_{P_2}(x),f_{P_3}(x),f_{P_4}(x),f_{P_5}(x)]$ the  least common multiple of $f_{P_1}(x),...,f_{P_5}(x)$, from (\ref{f-eq-1}) we see that $m(x)$ is just a product of basis factors, i.e.,
\begin{equation}\label{ff-0}m(x)=\prod_{p(x)\in S}p(x)=x(x^2-1)(x^2-2)(x^2-x-1)(x^2+x-1)(x^2-3).\end{equation} From now on, we assume that  $T \supset K_{1,3} $ is a quadratic starlike tree with the center vertex $u$ such that $T-u=n_1P_1\cup n_2P_2\cup n_3P_3\cup n_4P_4\cup n_5 P_5$, and
\begin{equation}\label{ff-1}f_T(x)=x^{t_1}(x^2-1)^{t_2}(x^2-2)^{t_3}((x^2-x-1)(x^2+x-1))^{t_4}(x^2-3)^{t_5}g(x),\end{equation}
where  $g(x)=x^2-c$ or $(x^2-ax+b)(x^2+ax+b)$, and  define two functions:
$$\begin{array}{ll}
&t_{(n_1,n_2,n_3,n_4,n_5)}(x)=m(x)[x-(n_1\frac{f_{P_0}(x)}{f_{P_1}(x)}+n_2\frac{f_{P_1}(x)}{f_{P_2}(x)}+n_3\frac{f_{P_2}(x)}{f_{P_3}(x)}
+n_4\frac{f_{P_3}(x)}{f_{P_4}(x)}+n_5\frac{f_{P_4}(x)}{f_{P_5}(x)})],\\
&u_{(z_1,z_2,z_3,z_4,z_5)}(x)=x^{z_1}(x^2-1)^{z_2}(x^2-2)^{z_3}((x^2-x-1)(x^2+x-1))^{z_4}(x^2-3)^{z_5}g(x).
\end{array}
$$

\begin{lem}\label{lem-z}
 Under above assumption,  if $z_1=t_1+1-n_1-n_3-n_5$, $z_2=t_2+1-n_2-n_5$ and $z_i=t_i+1-n_i$ for $i=3,4,5$, then
\begin{equation}\label{Eq-t-u}t_{(n_1,n_2,n_3,n_4,n_5)}(x)=u_{(z_1,z_2,z_3,z_4,z_5)}(x),\end{equation}
\begin{equation}\label{eq-6-10}z_1+2z_2+2z_3+4z_4+2z_5+\partial(g(x))=12,\end{equation}
\begin{equation}\label{eq-z}
\left\{\begin{array}{ll}
z_1=t_1+1-n_1-n_3-n_5=\left\{\begin{array}{ll}
0,&\mbox{if and only if $n_1+n_3+n_5\geq 1$;}\\
2,&\mbox{if and only if $n_1+n_3+n_5=0$,}\\
\end{array}
\right.\\
z_2=t_2+1-n_2-n_5=\left\{\begin{array}{ll}
0,&\mbox{if and only if $n_2+n_5\geq 1$;}\\
\mbox{$1$ or $2$},&\mbox{if and only if $n_2+n_5=0$,}\\\end{array}
\right.\\
z_3=t_3+1-n_3=\left\{\begin{array}{ll}
0,&\mbox{if and only if $n_3\geq 1$;} \\
\mbox{$1$ or $2$},&\mbox{if and only if $n_3=0$,}\end{array}\right. \\
z_4=t_4+1-n_4=\left\{\begin{array}{ll}
0,&\mbox{if and only if $n_4\geq 1$;} \\
\mbox{$1$ or $2$},&\mbox{if and only if $n_4=0$,}\end{array}\right. \\
z_5=t_5+1-n_5=\left\{\begin{array}{ll}
0,&\mbox{if and only if $n_5\geq 1$;} \\
\mbox{$1$ or $2$},&\mbox{if and only if $n_5=0$.}\end{array}\right. \\
\end{array}\right.
\end{equation} 
\end{lem}

\begin{proof}
First, by Lemma \ref{starlike-poly} we have
$$\begin{array}{ll}\!\!f_T(x)
\!\!&\!\!\!\!=\!\!x\prod_{i=1}^5f^{n_i}_{P_i}(x)-\sum_{i=1}^5[n_if_{P_{i-1}}(x)f^{n_i-1}_{P_i}(x)\prod_{j\not=i}f^{n_j}_{P_j}(x)]\\
&\!\!\!\!=\!\!xf_{T-u}(x)\!-\!(n_1f_{P_0}(x)\frac{f_{T-u}(x)}{f_{P_1}(x)}\!+\!n_2f_{P_1}(x)\frac{f_{T-u}(x)}{f_{P_2}(x)}\!+\!n_3f_{P_2}(x)\frac{f_{T-u}(x)}{f_{P_3}(x)}
\!+\!n_4f_{P_3}(x)\frac{f_{T-u}(x)}{f_{P_4}(x)}\!+\!n_5f_{P_4}(x)\frac{f_{T-u}(x)}{f_{P_5}(x)})\\
&\!\!\!\!=\!\!f_{T-u}(x)[x-(n_1\frac{f_{P_0}(x)}{f_{P_1}(x)}+n_2\frac{f_{P_1}(x)}{f_{P_2}(x)}+n_3\frac{f_{P_2}(x)}{f_{P_3}(x)}
+n_4\frac{f_{P_3}(x)}{f_{P_4}(x)}+n_5\frac{f_{P_4}(x)}{f_{P_5}(x)})],\end{array}
$$
which gives that $\frac{m(x)f_{T}(x)}{f_{T-u}(x)}=t_{(n_1,n_2,n_3,n_4,n_5)}(x)$.
On the other hand,  notice that
$$\begin{array}{ll}
f_{T-u}(x)&=f_{P_1}^{n_1}(x)f_{P_2}^{n_2}(x)f_{P_3}^{n_3}(x)f_{P_4}^{n_4}(x)f_{P_5}^{n_5}(x)\\
&=x^{n_1+n_3+n_5}(x^2-1)^{n_2+n_5}(x^2-2)^{n_3}((x^2-x-1)(x^2+x-1))^{n_4}(x^2 - 3)^{n_5}.\\
\end{array}
$$
From (\ref{ff-0}) and (\ref{ff-1}) we have
 $$\begin{array}{ll}\frac{m(x)f_{T}(x)}{f_{T-u}(x)}
&\!\!\!\!=\!\!x^{t_1+1-n_1-n_3-n_5}(x^2\!\!-\!\!1)^{t_2+1-n_2-n_5}(x^2\!\!-\!\!2)^{t_3+1-n_3}((x^2\!\!-\!\!x\!\!-\!\!1)(x^2\!\!+\!\!x\!\!-\!\!1))^{t_4+1-n_4}(x^2\!\!-\!\!3)^{t_5+1-n_5}g(x)\\
&\!\!\!\!=\!\!u_{(z_1,z_2,z_3,z_4,z_5)}(x).\end{array}$$
Therefore, Eq. (\ref{Eq-t-u}) follows.

Next, by considering the degree of Eq. (\ref{Eq-t-u}), we have $\partial(u_{(z_1,z_2,z_3,z_4,z_5)}(x))=z_1+2z_2+2z_3+4z_4+2z_5+\partial(g(x))=\partial(t_{(n_1,n_2,n_3,n_4,n_5)}(x))=\partial(m(x))+1=12$.

At last, we get the expression of $t_1$ by Corollary \ref{0-mul} and that of $t_2$, $t_3$, $t_4$ and $t_5$, respectively,  by Lemma \ref{starlike-mul}.
$$\begin{array}{ll}
t_1=\left\{\begin{array}{ll}
n_1+n_3+n_5-1,&\mbox{if and only if $n_1+n_3+n_5\geq 1$;}\\
1,&\mbox{if and only if $n_1+n_3+n_5=0$,}
\end{array}\right.\\
  t_2=\left\{\begin{array}{ll}
n_2+n_5-1,&\mbox{if and only if $n_2+n_5\geq 1$;}\\
\mbox{$0$ or $1$},&\mbox{if and only if $n_2+n_5=0$,}\\\end{array}\right.\\
t_3=\left\{\begin{array}{ll}
n_3-1,&\mbox{if and only if $n_3\geq 1$;}\\
\mbox{$0$ or $1$},&\mbox{if and only if $n_3=0$,}\end{array}\right.\\
 t_4=\left\{\begin{array}{ll}
n_4-1,&\mbox{if and only if $n_4\geq 1$;} \\
\mbox{$0$ or $1$},&\mbox{if and only if $n_4=0$,}\end{array}\right.\\
t_5=\left\{\begin{array}{ll}
n_5-1,&\mbox{if and only if $n_5\geq 1$;} \\
\mbox{$0$ or $1$},&\mbox{if and only if $n_5=0$,}\end{array}\right. \\
\end{array}$$
which leads to (\ref{eq-z}).

We complete this proof.
\end{proof}
\begin{remark}
We call Eq. (\ref{Eq-t-u}) and Eq. (\ref{eq-6-10}) the character equation and the parameter equation of $f_T(x)$, respectively, and (\ref{eq-z}) can be viewed as the restriction condition of Eq. (\ref{eq-6-10}).  
It is clear that Eq. (\ref{eq-6-10}) has finite  solutions since $0\le z_i\le 2$. In fact,  there also exist some relations among  $z_1$, $z_2$, $z_3$, $z_4$ and $z_5$. For example, from (\ref{eq-z}) we see that $z_1=z_2=0$ if $z_5=0$.  In terms of the restriction condition, we will identify  which $n_i$ equals to zero in $T-u=n_1P_1\cup n_2P_2\cup n_3P_3\cup n_4P_4\cup n_5 P_5$, and thus will simplify the character equation.
\end{remark}

\subsection{The quadratic starlike trees of form (I)}

In this subsection, we will determine  quadratic starlike tree  $T$ of  form (I). For a quadratic starlike tree   $T \supset K_{1,3} $ with the center vertex $u$, by Lemma \ref{starlike-two-path} and \ref{starlike-two-path-distinct} we have $T-u=n_1P_1\cup n_2P_2\cup n_3P_3\cup n_4P_4\cup n_5 P_5\cup P_k$ for some $k> 5$ or $T-u=n_1P_1\cup n_2P_2\cup n_3P_3\cup n_4P_4\cup n_5 P_5$. The following  Lemma \ref{starlike-P6} shows that the former does not occur.
\begin{lem}\label{starlike-P6}
Let  $T \supset K_{1,3} $ be a quadratic starlike tree with the center vertex $u$. If $f_{T}(x)$ is of  form (I) then $T-u$ can not  contain  any path $P_k$ with   $k>5$.
\end{lem}
\begin{proof}
Suppose to the contrary that $T-u$ contains exactly one path $P_k$ for some $k > 5$.
Since $f_T(x)$ has form (I), we have $\lambda_2(T)\leq\sqrt{3}$ by Corollary \ref{starlike-the-third-largest}(1). Thus $U_k$ described in Table \ref{table-forbidden-1} is a forbidden subgraph since $\lambda_2(U_k)>\sqrt{3}$ for  $k=6,7,8$. It implies that $T-u$ contains no any path $P_l$ for $l\geq8$, since otherwise $d_T(u)=2$ and $T$ will be a path $P_{j}$ for $j\ge 2+l\ge10$, which contradicts Theorem \ref{thm-path-cyc-1}. Thus $k=6$ or $7$. Since $U_{6}$ and $U_{7}$ are forbidden subgraphs of $T$, we deduce that
$$
d_T(u)=\left\{\begin{array}{ll}3,4,& \mbox{if $T-u$ contains $P_6$; }\\
3,& \mbox{if $T-u$ contains $P_7$. }\\
\end{array}\right.$$
Let $T_{1,0,0,1,0,1}$, $T_{0,1,1,0,0,1}$ and  $T_{2,1,0,0,0,1}$, shown  in Table \ref{table-forbidden-1}, be the graphs such that $T-u$ contains $P_6$, and $T_{1,1,0,0,0,0,1}$, shown  in Table \ref{table-forbidden-1}, be the graph such that $T-u$ contains $P_7$. By direct calculation, we know that they    are all forbidden subgraphs of $T$ since their second largest eigenvalues greater than $\sqrt{3}$ ( see  Table \ref{table-forbidden-1} ). Now let $F_i$ be the graphs shown  in Table \ref{table-subgraph} for $i=1,2,...,6$. Therefore,
we deduce that
$$T=\left\{\begin{array}{ll} F_1, F_2, F_3 \mbox{ or $F_4$, }& \mbox{ if $d_T(u)=3$ and $T-u$ contains $P_6$;}\\
  F_5, & \mbox{ if $d_T(u)=4$ and $T-u$ contains $P_6$;}\\
    F_6, & \mbox{ if $d_T(u)=3$ and $T-u$ contains $P_7$.}\\
\end{array}\right.
$$
By direct computation we can find the characteristic polynomials of   $F_1, F_2, F_3,F_4$, $F_5$ and  $F_6$ that  are listed in Table \ref{table-subgraph}, from which we see that they are not quadratic. Therefore, $T-u$  can't contain a path $P_k$ for $k=6,7$.

Summarizing the above discussions, we get our result.
\end{proof}
\begin{table}[h]
\scriptsize
\caption{\small Forbidden subgraphs of form (I)}
\centering
\renewcommand\arraystretch{1.2}
\begin{tabular*}{13.6cm}{m{10pt}|m{75pt}|m{40pt}|m{25pt}|m{75pt}|m{30pt}|m{45pt}}
\hline
 $k$ & Forbidden subgraph $U_k$& $\lambda_2(U_k)$ & $d_T(u)$ & Forbidden subgraph $H$ & $\lambda_2(H)$ & $T$\\\hline
&&  &\multirow{2}*{$3$}&
\unitlength 0.10mm 
\linethickness{0.4pt}
\ifx\plotpoint\undefined\newsavebox{\plotpoint}\fi
\centering
\begin{picture}(47.9,69)
\put(19.6,65.3){\circle*{4}}
\put(0.0,45.7){\circle*{4}}
\qbezier(19.6,65.3)(9.8,55.5)(0.0,45.7)
\put(15.2,45.7){\circle*{4}}
\qbezier(19.6,65.3)(17.4,55.5)(15.2,45.7)
\put(37.7,47.1){\circle*{4}}
\qbezier(19.6,65.3)(28.6,56.2)(37.7,47.1)
\put(37.7,8.7){\circle*{4}}
\qbezier(37.7,18.9)(37.7,13.8)(37.7,8.7)
\put(37.7,26.1){\oval(16.0,51.5)}\qbezier(37.7,47.1)(37.7,40.6)(37.7,34.1)
\qbezier[4](37.7,34.1)(37.7,26.5)(37.7,18.9)
\put(47.9,29.7){\makebox(0,0)[tl]{$P_6$}}
\put(15.2,31.9){\circle*{4}}
\qbezier(15.2,45.7)(15.2,38.8)(15.2,31.9)
\put(15.2,17.4){\circle*{4}}
\qbezier(15.2,31.9)(15.2,24.7)(15.2,17.4)
\put(15.2,4.4){\circle*{4}}
\qbezier(15.2,17.4)(15.2,10.9)(15.2,4.4)
\end{picture}
 &$1.7531$&  \multirow{2}*{$F_1,F_2,F_3,F_4$ }\\
\cline{5-6}
 $6$ &\unitlength 0.10mm 
\linethickness{0.4pt}
\ifx\plotpoint\undefined\newsavebox{\plotpoint}\fi
\centering
\begin{picture}(84.8,80)
\put(56.6,73.2){\circle*{4}}
\put(29.0,45.7){\circle*{4}}
\qbezier(56.6,73.2)(42.8,59.5)(29.0,45.7)
\put(0.0,46.4){\circle*{4}}
\qbezier(56.6,73.2)(28.3,59.8)(0.0,46.4)
\put(42.8,45.7){\circle*{4}}
\qbezier(56.6,73.2)(49.7,59.5)(42.8,45.7)
\put(13.8,46.4){\circle*{4}}
\qbezier(56.6,73.2)(35.2,59.8)(13.8,46.4)
\put(74.7,47.1){\circle*{4}}
\qbezier(56.6,73.2)(65.6,60.2)(74.7,47.1)
\put(74.7,8.7){\circle*{4}}
\qbezier(74.7,18.9)(74.7,13.8)(74.7,8.7)
\put(74.7,26.1){\oval(16.0,51.5)}\qbezier(74.7,47.1)(74.7,40.6)(74.7,34.1)
\qbezier[4](74.7,34.1)(74.7,26.5)(74.7,18.9)
\put(84.8,29.7){\makebox(0,0)[tl]{$P_6$}}
\end{picture}&$1.7592$&  &
 \unitlength 0.10mm 
\linethickness{0.4pt}
\ifx\plotpoint\undefined\newsavebox{\plotpoint}\fi
\centering
\begin{picture}(47.9,69)
\put(19.6,65.3){\circle*{4}}
\put(0.0,45.7){\circle*{4}}
\qbezier(19.6,65.3)(9.8,55.5)(0.0,45.7)
\put(15.2,45.7){\circle*{4}}
\qbezier(19.6,65.3)(17.4,55.5)(15.2,45.7)
\put(37.7,47.1){\circle*{4}}
\qbezier(19.6,65.3)(28.6,56.2)(37.7,47.1)
\put(37.7,8.7){\circle*{4}}
\qbezier(37.7,18.9)(37.7,13.8)(37.7,8.7)
\put(37.7,26.1){\oval(16.0,51.5)}\qbezier(37.7,47.1)(37.7,40.6)(37.7,34.1)
\qbezier[4](37.7,34.1)(37.7,26.5)(37.7,18.9)
\put(47.9,29.7){\makebox(0,0)[tl]{$P_6$}}
\put(15.2,31.9){\circle*{4}}
\qbezier(15.2,45.7)(15.2,38.8)(15.2,31.9)
\put(15.2,17.4){\circle*{4}}
\qbezier(15.2,31.9)(15.2,24.7)(15.2,17.4)
\put(0.0,31.9){\circle*{4}}
\qbezier(0.0,45.7)(0.0,38.8)(0.0,31.9)
\end{picture}
 & $1.7469$\\
\cline{4-7}
&&&$4$
 & \unitlength 0.09mm 
\linethickness{0.4pt}
\ifx\plotpoint\undefined\newsavebox{\plotpoint}\fi
\centering
\begin{picture}(71.1,78)
\put(42.8,73.2){\circle*{4}}
\put(15.2,45.7){\circle*{4}}
\qbezier(42.8,73.2)(29.0,59.5)(15.2,45.7)
\put(29.0,45.7){\circle*{4}}
\qbezier(42.8,73.2)(35.9,59.5)(29.0,45.7)
\put(0.0,46.4){\circle*{4}}
\qbezier(42.8,73.2)(21.4,59.8)(0.0,46.4)
\put(60.9,47.1){\circle*{4}}
\qbezier(42.8,73.2)(51.8,60.2)(60.9,47.1)
\put(60.9,8.7){\circle*{4}}
\qbezier(60.9,18.9)(60.9,13.8)(60.9,8.7)
\put(60.9,26.1){\oval(16.0,51.5)}\qbezier(60.9,47.1)(60.9,40.6)(60.9,34.1)
\qbezier[4](60.9,34.1)(60.9,26.5)(60.9,18.9)
\put(71.1,29.7){\makebox(0,0)[tl]{$P_6$}}
\put(29.0,26.8){\circle*{4}}
\qbezier(29.0,45.7)(29.0,36.3)(29.0,26.8)
\end{picture}
 & $1.7482$ & ~~~~$F_5$~~~~ \\\hline
 $7$ &
\unitlength 0.10mm 
\linethickness{0.4pt}
\ifx\plotpoint\undefined\newsavebox{\plotpoint}\fi
\centering
\begin{picture}(55.1,69)
\put(26.8,65.3){\circle*{4}}
\put(11.6,46.4){\circle*{4}}
\qbezier(26.8,65.3)(19.2,55.8)(11.6,46.4)
\put(22.5,45.7){\circle*{4}}
\qbezier(26.8,65.3)(24.7,55.5)(22.5,45.7)
\put(45.0,47.1){\circle*{4}}
\qbezier(26.8,65.3)(35.9,56.2)(45.0,47.1)
\put(45.0,8.7){\circle*{4}}
\qbezier(45.0,18.9)(45.0,13.8)(45.0,8.7)
\put(45.0,26.1){\oval(16.0,51.5)}\qbezier(45.0,47.1)(45.0,40.6)(45.0,34.1)
\qbezier[4](45.0,34.1)(45.0,26.5)(45.0,18.9)
\put(55.1,29.7){\makebox(0,0)[tl]{$P_7$}}
\put(0.0,46.4){\circle*{4}}
\qbezier(26.8,65.3)(13.4,55.8)(0.0,46.4)
\end{picture}
 & $1.7943$ & $3$ &
\unitlength 0.10mm 
\linethickness{0.4pt}
\ifx\plotpoint\undefined\newsavebox{\plotpoint}\fi
\centering
\begin{picture}(43.5,69)
\put(15.2,65.3){\circle*{4}}
\put(0.0,46.4){\circle*{4}}
\qbezier(15.2,65.3)(7.6,55.8)(0.0,46.4)
\put(10.9,45.7){\circle*{4}}
\qbezier(15.2,65.3)(13.1,55.5)(10.9,45.7)
\put(33.4,47.1){\circle*{4}}
\qbezier(15.2,65.3)(24.3,56.2)(33.4,47.1)
\put(33.4,8.7){\circle*{4}}
\qbezier(33.4,18.9)(33.4,13.8)(33.4,8.7)
\put(33.4,26.1){\oval(16.0,51.5)}\qbezier(33.4,47.1)(33.4,40.6)(33.4,34.1)
\qbezier[4](33.4,34.1)(33.4,26.5)(33.4,18.9)
\put(43.5,29.7){\makebox(0,0)[tl]{$P_7$}}
\put(10.9,33.4){\circle*{4}}
\qbezier(10.9,45.7)(10.9,39.5)(10.9,33.4)
\end{picture}
& $1.7692$ & ~~~$F_6$~~~\\\hline
$8$&
\unitlength 0.10mm 
\linethickness{0.4pt}
\ifx\plotpoint\undefined\newsavebox{\plotpoint}\fi
\centering
\begin{picture}(43.5,69)
\put(15.2,65.3){\circle*{4}}
\put(0.0,46.4){\circle*{4}}
\qbezier(15.2,65.3)(7.6,55.8)(0.0,46.4)
\put(10.9,45.7){\circle*{4}}
\qbezier(15.2,65.3)(13.1,55.5)(10.9,45.7)
\put(33.4,47.1){\circle*{4}}
\qbezier(15.2,65.3)(24.3,56.2)(33.4,47.1)
\put(33.4,8.7){\circle*{4}}
\qbezier(33.4,18.9)(33.4,13.8)(33.4,8.7)
\put(33.4,26.1){\oval(16.0,51.5)}\qbezier(33.4,47.1)(33.4,40.6)(33.4,34.1)
\qbezier[4](33.4,34.1)(33.4,26.5)(33.4,18.9)
\put(43.5,29.7){\makebox(0,0)[tl]{$P_8$}}
\end{picture}
&$1.7820$& $2$ &&&\\\hline
\end{tabular*}\label{table-forbidden-1}
\end{table}

\begin{table}[h]
\scriptsize
\caption{\small Possible graphs  of form (I)}
\centering
\begin{tabular*}{11cm}{m{15pt}|m{60pt}|m{155pt}|m{25pt}}
\hline
$F_i$ &~~\ ~~~~graph~~~~~& the factorization of $f_{F_i}(x)$&  quadratic\\\hline
$F_1$ &
\unitlength 0.08mm 
\linethickness{0.4pt}
\ifx\plotpoint\undefined\newsavebox{\plotpoint}\fi
\centering
\begin{picture}(43.5,69)
\put(15.2,65.3){\circle*{4}}
\put(0.0,46.4){\circle*{4}}
\qbezier(15.2,65.3)(7.6,55.8)(0.0,46.4)
\put(10.9,45.7){\circle*{4}}
\qbezier(15.2,65.3)(13.1,55.5)(10.9,45.7)
\put(33.4,47.1){\circle*{4}}
\qbezier(15.2,65.3)(24.3,56.2)(33.4,47.1)
\put(33.4,8.7){\circle*{4}}
\qbezier(33.4,18.9)(33.4,13.8)(33.4,8.7)
\put(33.4,26.1){\oval(16.0,51.5)}\qbezier(33.4,47.1)(33.4,40.6)(33.4,34.1)
\qbezier[4](33.4,34.1)(33.4,26.5)(33.4,18.9)
\put(43.5,29.7){\makebox(0,0)[tl]{$P_6$}}
\end{picture}
&$
x(x^8 - 8x^6 + 20x^4 - 16x^2 + 2)$& No\\\hline
$F_2$ &
\unitlength 0.08mm 
\linethickness{0.4pt}
\ifx\plotpoint\undefined\newsavebox{\plotpoint}\fi
\centering
\begin{picture}(43.5,69)
\put(15.2,65.3){\circle*{4}}
\put(0.0,46.4){\circle*{4}}
\qbezier(15.2,65.3)(7.6,55.8)(0.0,46.4)
\put(10.9,45.7){\circle*{4}}
\qbezier(15.2,65.3)(13.1,55.5)(10.9,45.7)
\put(33.4,47.1){\circle*{4}}
\qbezier(15.2,65.3)(24.3,56.2)(33.4,47.1)
\put(33.4,8.7){\circle*{4}}
\qbezier(33.4,18.9)(33.4,13.8)(33.4,8.7)
\put(33.4,26.1){\oval(16.0,51.5)}\qbezier(33.4,47.1)(33.4,40.6)(33.4,34.1)
\qbezier[4](33.4,34.1)(33.4,26.5)(33.4,18.9)
\put(43.5,29.7){\makebox(0,0)[tl]{$P_6$}}
\put(10.9,36.3){\circle*{4}}
\qbezier(10.9,45.7)(10.9,41.0)(10.9,36.3)
\end{picture}
&$x^{10} - 9x^8 + 27x^6 - 31x^4 + 12x^2 - 1$& No\\\hline
$F_3$ &
\unitlength 0.08mm 
\linethickness{0.4pt}
\ifx\plotpoint\undefined\newsavebox{\plotpoint}\fi
\centering
\begin{picture}(43.5,69)
\put(15.2,65.3){\circle*{4}}
\put(0.0,46.4){\circle*{4}}
\qbezier(15.2,65.3)(7.6,55.8)(0.0,46.4)
\put(10.9,45.7){\circle*{4}}
\qbezier(15.2,65.3)(13.1,55.5)(10.9,45.7)
\put(33.4,47.1){\circle*{4}}
\qbezier(15.2,65.3)(24.3,56.2)(33.4,47.1)
\put(33.4,8.7){\circle*{4}}
\qbezier(33.4,18.9)(33.4,13.8)(33.4,8.7)
\put(33.4,26.1){\oval(16.0,51.5)}\qbezier(33.4,47.1)(33.4,40.6)(33.4,34.1)
\qbezier[4](33.4,34.1)(33.4,26.5)(33.4,18.9)
\put(43.5,29.7){\makebox(0,0)[tl]{$P_6$}}
\put(10.9,36.3){\circle*{4}}
\qbezier(10.9,45.7)(10.9,41.0)(10.9,36.3)
\put(10.9,26.8){\circle*{4}}
\qbezier(10.9,36.3)(10.9,31.5)(10.9,26.8)
\end{picture}
&$x(x - 1)(x + 1)(x^2 - 3)(x^6 - 6x^4 + 8x^2 - 1)$& No\\\hline
$F_4$ &
\unitlength 0.08mm 
\linethickness{0.4pt}
\ifx\plotpoint\undefined\newsavebox{\plotpoint}\fi
\centering
\begin{picture}(43.5,69)
\put(15.2,65.3){\circle*{4}}
\put(0.0,46.4){\circle*{4}}
\qbezier(15.2,65.3)(7.6,55.8)(0.0,46.4)
\put(10.9,45.7){\circle*{4}}
\qbezier(15.2,65.3)(13.1,55.5)(10.9,45.7)
\put(33.4,47.1){\circle*{4}}
\qbezier(15.2,65.3)(24.3,56.2)(33.4,47.1)
\put(33.4,8.7){\circle*{4}}
\qbezier(33.4,18.9)(33.4,13.8)(33.4,8.7)
\put(33.4,26.1){\oval(16.0,51.5)}\qbezier(33.4,47.1)(33.4,40.6)(33.4,34.1)
\qbezier[4](33.4,34.1)(33.4,26.5)(33.4,18.9)
\put(43.5,29.7){\makebox(0,0)[tl]{$P_6$}}
\put(10.9,36.3){\circle*{4}}
\qbezier(10.9,45.7)(10.9,41.0)(10.9,36.3)
\put(0.0,36.3){\circle*{4}}
\qbezier(0.0,46.4)(0.0,41.3)(0.0,36.3)
\end{picture}
&$x(x - 1)(x + 1)(x^2 - 3)(x^6 - 6x^4 + 8x^2 - 2)$& No\\\hline
$F_5$ &
\unitlength 0.08mm 
\linethickness{0.4pt}
\ifx\plotpoint\undefined\newsavebox{\plotpoint}\fi
\centering
\begin{picture}(53.7,69)
\put(25.4,65.3){\circle*{4}}
\put(10.2,46.4){\circle*{4}}
\qbezier(25.4,65.3)(17.8,55.8)(10.2,46.4)
\put(21.0,45.7){\circle*{4}}
\qbezier(25.4,65.3)(23.2,55.5)(21.0,45.7)
\put(43.5,47.1){\circle*{4}}
\qbezier(25.4,65.3)(34.4,56.2)(43.5,47.1)
\put(43.5,8.7){\circle*{4}}
\qbezier(43.5,18.9)(43.5,13.8)(43.5,8.7)
\put(43.5,26.1){\oval(16.0,51.5)}\qbezier(43.5,47.1)(43.5,40.6)(43.5,34.1)
\qbezier[4](43.5,34.1)(43.5,26.5)(43.5,18.9)
\put(53.7,29.7){\makebox(0,0)[tl]{$P_6$}}
\put(0.0,46.4){\circle*{4}}
\qbezier(25.4,65.3)(12.7,55.8)(0.0,46.4)
\end{picture}
&$x^2(x^2 - 3)(x^6 - 6x^4 + 7x^2 - 1)$& No\\\hline
$F_6$ &
\unitlength 0.08mm 
\linethickness{0.4pt}
\ifx\plotpoint\undefined\newsavebox{\plotpoint}\fi
\centering
\begin{picture}(43.5,69)
\put(15.2,65.3){\circle*{4}}
\put(0.0,46.4){\circle*{4}}
\qbezier(15.2,65.3)(7.6,55.8)(0.0,46.4)
\put(10.9,45.7){\circle*{4}}
\qbezier(15.2,65.3)(13.1,55.5)(10.9,45.7)
\put(33.4,47.1){\circle*{4}}
\qbezier(15.2,65.3)(24.3,56.2)(33.4,47.1)
\put(33.4,8.7){\circle*{4}}
\qbezier(33.4,18.9)(33.4,13.8)(33.4,8.7)
\put(33.4,26.1){\oval(16.0,51.5)}\qbezier(33.4,47.1)(33.4,40.6)(33.4,34.1)
\qbezier[4](33.4,34.1)(33.4,26.5)(33.4,18.9)
\put(43.5,29.7){\makebox(0,0)[tl]{$P_7$}}
\end{picture}
&$x^2(x^2 - 3)(x^6 - 6x^4 + 9x^2 - 3)$& No\\\hline
\end{tabular*}\label{table-subgraph}
\end{table}

\begin{thm}\label{thm-quadratic-starlike-I}
Let  $T \supset K_{1,3} $ be a quadratic starlike tree. Then $f_{T}(x)$ is of  form (I) if and only if  $T$ is one of the four graphs that listed in Table \ref{table-quadratic-I}.
\end{thm}
\begin{proof}
According to Lemma \ref{starlike-P6},  $T-u=n_1P_1\cup n_2P_2\cup n_3P_3\cup n_4P_4\cup n_5P_5$ where $n_1+n_2+\cdots+n_5\geq3$.
 From Lemma \ref{lem-z}, the character equation Eq. (\ref{Eq-t-u}) becomes
 \begin{equation}\label{h-eq-1}
\begin{array}{ll}
&xm(x)-(n_1(x^2-1)(x^2-2)(x^2-x-1)(x^2+x-1)(x^2 - 3)\\
&+n_2x^2(x^2\!-\!2)(x^2\!-\!x\!-\!1)(x^2\!+\!x\!-\!1)(x^2\! - \! 3)+n_3(x^2\!-\!1)^2(x^2\!-\!x\!-\!1)(x^2\!+\!x\!-\!1)(x^2\! - \!3)\\
&+n_4x^2(x^2\!-\!1)(x^2\!-\!2)^2(x^2\!- \!3)+n_5(x^2\!-\!2)(x^2\!-\!x\!-\!1)^2(x^2\!+\!x\!-\!1)^2)\\
&=x^{z_1}(x^2-1)^{z_2}(x^2-2)^{z_3}((x^2-x-1)(x^2+x-1))^{z_4}(x^2-3)^{z_5}(x^2-c),\\
\end{array}
\end{equation}
and the corresponding parameter equation becomes
\begin{equation}\label{z-eq-10}
z_1+2z_2+2z_3+4z_4+2z_5+2=12.
\end{equation}
Using the restriction condition (\ref{eq-z}), 
it is routine to  find all the fifteen solutions  $(z_1,z_2,z_3,z_4,z_5)$ of  (\ref{z-eq-10}) which we  list in Table \ref{table-quadratic-I-invalid} and Table \ref{table-quadratic-I}. From these solutions   we can get $(n_1,n_2,n_3,n_4,n_5)$ and then determine  all the quadratic starlike trees of form (I). In what follows we prove our result by two steps.

{\flushleft\bf Step 1. } All the solutions $(z_1,z_2,z_3,z_4,z_5)$ listed   in  Table \ref{table-quadratic-I-invalid} are  invalid.

We now confirm our conclusion in detail by taking the solution $(z_1,z_2,z_3,z_4,z_5)=(2,2,1,0,1)$ at the first row of Table \ref{table-quadratic-I-invalid}.  According to the restriction (\ref{eq-z}) in Lemma \ref{lem-z}, we have $(n_1,n_2,n_3,n_4,n_5)=(0,0,0,n_4,0)$. Thus, the Eq. (\ref{h-eq-1}) becomes
$xm(x)-n_4x^2(x^2-1)(x^2-2)^2(x^2 - 3)
=x^{2}(x^2-1)^{2}(x^2-2)(x^2-3)(x^2-c)$.
By deleting $(x^2-1)(x^2-2)(x^2-3)$ on two sides, we get the simplification of the character equation $$(x^2-x-1)(x^2+x-1)-n_4(x^2-2)=(x^2-1)(x^2-c).$$
By comparing the coefficients, we get the restriction condition
$$\left\{\begin{array}{ll}
n_4+3=c+1\\
2n_4+1=c\\
\end{array}\right.
\Longrightarrow
\left\{\begin{array}{ll}
n_4=1\\c=3
\end{array}\right.$$
However, $1=n_1+n_2+n_3+n_4+n_5 =d_T(u)\geq3$, a contradiction.  The solution is invalid.

As the same arguments as above, from the solutions  $(z_1,z_2,z_3,z_4,z_5)$ we can determine $(n_1,n_2,n_3,n_4,n_5)$, which leads to  a  simplification of the character equation, by comparing the  coefficients, we find the restriction conditions which are conflicting. In this way, one can verify that  all these solutions are invalid (to see 2-5 columns of Table \ref{table-quadratic-I-invalid} for details ).
\begin{table}[htpb]
\scriptsize
\caption{\small Invalid solutions of quadratic starlike tree of form (I)}
\centering
\renewcommand\arraystretch{1.3}
\begin{tabular*}{15cm}{m{3pt}|m{50pt}|m{50pt}|m{180pt}|m{85pt}}
\hline
 &$(z_1,z_2,z_3,z_4,z_5)$ &$(n_1,n_2,n_3,n_4,n_5)$& simplification of the character equation&restriction condition\\\hline
1 & $(2,2,1,0,1)$ & $(0,0,0,n_4,0)$ & $(x^2-x-1)(x^2+x-1)-n_4(x^2-2)=(x^2-1)(x^2-c)$& $n_4=1$
\\\hline
2 &
$(2,1,1,0,2)$ & $(0,0,0,n_4,0)$  & $(x^2-x-1)(x^2+x-1)-n_4(x^2-2)=(x^2-3)(x^2-c)$&$n_4=1$\\\hline
3 &
$(0,2,1,0,2)$ &  $(n_1,0,0,n_4,0)$ & $x^2(x^2-x-1)(x^2+x-1)-(n_1(x^2-x-1)(x^2+x-1)+n_4x^2(x^2-2))=(x^2-1)(x^2-3)(x^2-c)$&$n_1=0$, $n_4=1$\\\hline
4 &
$(0,0,2,1,1)$ & $(n_1,n_2,0,0,0)$  & $x^2(x^2\!-\!1)\!-\!(n_1(x^2\!-\!1)+\!n_2x^2)
=(x^2-2)(x^2-c)$&$n_1=2c$,  $n_2=1-c$\\\hline
5 &
$(0,0,1,1,2)$ & $(n_1 ,n_2,0,0,0)$  & $x^2(x^2-1)-(n_1(x^2-1)+n_2x^2)=(x^2-3)(x^2-c)$&$n_1=3c$, $n_2=2(1-c)$\\\hline
6 &
$(0,0,0,2,1)$ & $(n_1,n_2 ,n_3,0,0)$  & $x^2(x^2-1)(x^2-2)-(n_1(x^2-1)(x^2-2)
+n_2x^2(x^2\!-\!2)+n_3(x^2\!-\!1)^2
=(x^2-x-1)(x^2+x-1)(x^2-c)$&$n_1=n_2=c-1$, $n_3=2-c$\\\hline
7 &
$(0,1,0,1,2)$ & $(n_1,0,n_3,0,0)$   & $ x^2(x^2-2)-(n_1(x^2-2)
+n_3(x^2-1)=(x^2-3)(x^2-c)$&$n_1=2c-1$, $n_3=2-c $\\\hline
8 &$(0,2,2,0,1)$ & $(n_1,0,0,n_4,0)$ & $x^2(x^2-x-1)(x^2+x-1)-(n_1(x^2-x-1)(x^2+x-1)+n_4x^2(x^2\!-\!2)
=(x^2-1)(x^2-2)(x^2-c)$& $n_1=2$, $n_4=-1$\\\hline
9 &
$(2,0,2,0,2)$ & $(0,n_2,0,n_4,0)$  &$(x^2\!-\!1)(x^2\!-\!x\!-\!1)(x^2\!+\!x\!-\!1)-(n_2(x^2\!-\!x\!-\!1)(x^2\!+\!x\!-\!1)
+n_4(x^2\!-\!1)(x^2\!-\!2))
=(x^2-2)(x^2-3)(x^2-c)$ &$n_2=13$, $n_4=-\frac{29}{2}$\\\hline
10 &
$(0,1,2,0,2)$ & $(n_1 ,0,0,n_4,0)$  & $x^2(x^2-x-1)(x^2+x-1)-(n_1(x^2-x-1)(x^2+x-1)
+n_4x^2(x^2\!-\!2))
=(x^2-2)(x^2-3)(x^2-c)$&$n_1=2$, $n_4=\frac{1}{3}$\\\hline
11 &$(2,1,2,0,1)$ &$(0,0,0,n_4,0)$ & $(x^2-x-1)(x^2+x-1)-n_4(x^2\!-\!2)=(x^2-2)(x^2-c)$&no solution \\\hline
\end{tabular*}\label{table-quadratic-I-invalid}
\end{table}

{\flushleft\bf Step 2. } All  the solutions $(z_1,z_2,z_3,z_4,z_5)$ listed  in  Table \ref{table-quadratic-I} are valid.

By taking $(z_1,z_2,z_3,z_4,z_5)=(0,0,1,2,0)$ in 1th row of Table \ref{table-quadratic-I}, we get $n_3=n_4=0$ according to (\ref{eq-z}) in Lemma \ref{lem-z}. Thus, the Eq. (\ref{h-eq-1}) becomes
 $$\begin{array}{ll}
&xm(x)-(n_1(x^2-1)(x^2-2)(x^2-x-1)(x^2+x-1)(x^2-3)\\
&+n_2x^2(x^2-2)(x^2-x-1)(x^2+x-1)(x^2 - 3)
+n_5(x^2-2)(x^2-x-1)^2(x^2+x-1)^2)\\
&=(x^2-2)((x^2-x-1)(x^2+x-1))^2(x^2-c),
\end{array}$$
which can be  simplified as
$$x^2(x^2\!-\!1)(x^2\!-\!3)-(n_1(x^2\!-\!1)(x^2\!-\!3)+n_2x^2(x^2 \!-\! 3)
+n_5(x^2\!-\!x\!-\!1)(x^2\!+\!x\!-\!1))
\!=\!(x^2\!-\!x\!-\!1)(x^2\!+\!x\!-\!1)(x^2\!-\!c).$$
By comparing the coefficients, we get the restriction condition
$$\left\{\begin{array}{ll}
n_1+n_2+n_5+4=c+3;\\
4n_1+3n_2+3n_5+2=3c+1;\\
3n_1+n_5=c,\\
\end{array}\right.$$
which leads to $(n_1,n_2,n_3,n_4,n_5)=(1,1, 0,0,c-3)$ with  the simple condition $n_1=n_2=1$ and $n_5=c-3\ge 1$. Thus $T=T_{1,1,0,0,n_5}$ is a quadratic starlike tree  with characteristic polynomial
 $$\begin{array}{ll}
f_T(x)&=u_{(1,1,0,0,n_5)}(x)\frac{f_{T-u}(x)}{m(x)}\\
&=(x^2-2)((x^2-x-1)(x^2+x-1))^2(x^2-c)\frac{(x^2-3)^{n_5}(x^2-1)^{n_5+1}x^{n_5+1}}{m(x)}\\
&=x^{n_5}(x^2-1)^{n_5}(x^2-x-1)(x^2+x-1)(x^2-3)^{n_5-1}(x^2-c)\\
&=x^{n_5}(x^2-1)^{n_5}(x^2-x-1)(x^2+x-1)(x^2-3)^{n_5-1}(x^2-(n_5+3)).
\end{array}$$
Similarly, for the solutions  $(z_1,z_2,z_3,z_4,z_5)=(1,2,0,1,1), (2, 0,1,1,1)$ and $ (0,1,1,1,1)$ in 2-4 rows of Table \ref{table-quadratic-I}, we can simplify the  character equations respectively
$$\begin{array}{ll}
&x^2(x^2-2)-(n_1(x^2-2)+n_3(x^2-1))=(x^2-1)(x^2-c)\\
&(x^2-1)-n_2=(x^2-c)\\
&x^2-n_1=(x^2-c)\\
\end{array}\Longrightarrow \begin{array}{ll}
&n_1=1, n_3=c-2\ge2\\
&n_2=c-1\ge3\\
&n_1=c\ge4\\
\end{array}
$$
Thus, we can determine the quadratic starlike trees $T_{1,0,n_3}( n_3\ge 2), T_{0,n_2}( n_2\ge 3)$ and $T_{n_1}( n_1\ge 4)$ along with their characteristic polynomials listed in  Table \ref{table-quadratic-I}, respectively. Therefore, the above four  starlike trees are all quadratic of form (I).

We complete this proof.
\end{proof}

\begin{table}[h]
\scriptsize
\caption{\small Quadratic starlike trees  of form (I)}
\centering
\renewcommand\arraystretch{1.4}
\begin{tabular*}{15cm}{m{3pt}|m{50pt}|m{52pt}|m{65pt}|m{190pt}}
\hline
 &$(z_1,z_2,z_3,z_4,z_5)$ &$(n_1,n_2,n_3,n_4,n_5)$& ~~~~~~~$T$&~~~~~~$f_T(x)$\\\hline
1 &$(0,0,1,2,0)$ & $(1,1,0,0, c-3)$ & $T_{1,1,0,0,n_5}(n_5\geq1)$& $
x^{n_5}(x^2-1)^{n_5}(x^2-x-1)(x^2+x-1)(x^2-3)^{n_5-1}(x^2-(n_5+3))$\\\hline
2 &$(0,2,0,1,1)$ & $(1,0,c-2,0,0)$ & $T_{1,0,n_3}~~(n_3\geq2)$& $
x^{n_3}(x^2-1)(x^2-2)^{n_3-1}(x^2-(n_3+2))$\\\hline
3 &
$(2,0,1,1,1)$ & $(0,c-1,0, 0,0)$ & $T_{0,n_2}~~(n_2\geq3)$& $
x(x^2-1)^{n_2-1}(x^2-(n_2+1))$\\\hline
4 &
$(0,1,1,1,1)$ & $(c,0,0,0,0)$ & $T_{n_1}~~(n_1\geq4)$& $
x^{n_1-1}(x^2-n_1)$\\\hline
\end{tabular*}\label{table-quadratic-I}
\end{table}
\begin{remark}
When $n_1$ and $n_2+1$ are square numbers, the starlike graphs listed in 3th and 4th rows of Table \ref{table-quadratic-I} are  integral which is the main result in \cite{Watanabe}.
\end{remark}

\subsection{The quadratic starlike trees of form (II)}

In this subsection, we will determine  quadratic starlike tree  $T$ of  form (II). Unlike Lemma \ref{starlike-P6}, we can not directly show that $\lambda_2(T)\le \sqrt{3}$ for $T$ of form (II). To excluded $P_k(k\ge 6)$ in $T-u$ we have to use other method. First of all, we will introduce some notion, symbols and give some lemmas for  the preparations.   By simple calculations,
\begin{equation}\label{f-eq-1}\begin{array}{ll}
f_{P_1}(x)=x; \  \ f_{P_2}(x)=x^2-1; \ \ f_{P_3}(x)=x(x^2-2);\\
f_{P_4}(x)=(x^2-x-1)(x^2+x-1);\ \  f_{P_5}(x)=x(x - 1)(x + 1)(x^2 - 3);\\
f_{P_6}(x)=(x^3 - x^2 - 2x + 1)(x^3 + x^2 - 2x - 1);\\
f_{P_7}(x)=x(x^2 - 2)(x^4 -4x^2+2);\\
f_{P_8}(x)=(x - 1)(x + 1)(x^3 -3x + 1)(x^3 -3x - 1);\\
f_{P_9}(x)=x(x^2-x-1)(x^2+x-1)(x^4 -5x^2+5);\\
f_{P_{10}}(x)=(x^5+x^4-4x^3-3x^2 +3x + 1)(x^5-x^4-4x^3 +3x^2 +3x-1);\\
f_{P_{11}}(x)=x(x - 1)(x + 1)(x^2 - 2)(x^2 - 3)(x^4 -4x^2+1);\\
f_{P_{12}}(x)=(x^6-x^5-5x^4+4x^3+6x^2 -3x\! -\! 1)(x^6+x^5-5x^4-4x^3+6x^2 +3x\! - \! 1);\\
f_{P_{13}}(x)=x(x^3 - x^2 - 2x + 1)(x^3 + x^2 - 2x - 1)(x^6-7x^4+14x^2-7).
\end{array}
\end{equation}
Now we assume that $T-u=n_1P_1\cup n_2 P_2\cup\cdots\cup n_5P_5\cup P_k$, where $n_1+n_2+n_3+n_4+n_5\ge 2$ and $6\le k\le 13$. In this situation we specify $T$ as $T_k$  to related with $P_k$ and denote by $c_k(x)$ the product of all the irreducible factors of $f_{T_k-u}(x)$.  Then $f_{P_k}(x)|c_k(x)$, and $f_{P_i}(x)|c_k(x)$ where $P_i\in T_k-u-P_k$. Denote by $S_k$ the  set of basis factors in $c_{k}(x)$ and clearly $S_k\subseteq S$.  $c_{S_k}(x)=\prod_{p(x)\in S_k} p(x)$ is defined to be the elementary part of $c_{k}(x)$ and $q_k(x)=\frac{c_{k}(x)}{c_{S_k}(x)}$ the non-elementary part. Then $c_{k}(x)=c_{S_k}(x)\cdot q_k(x)$ and $q_k(x)|f_{P_k}(x)$ ( for instance, $q_6(x)=f_{P_6}(x)$ and $q_7(x)=(x^4 -4x^2+2)$, $x(x^2-2)|c_{S_7}(x)$ and so on ). By Theorem \ref{star-thm-1} and Lemma \ref{starlike-poly} we have
\begin{equation}\label{P-eq-3}\begin{array}{ll}
&x^{t_1}(x^2\!-\!1)^{t_2} (x^2\!-\!2)^{t_3}((x^2\!-\!x\!-\!1)(x^2\!+\!x\!-\!1))^{t_4}(x^2\!-\!3)^{t_5}(x^2\!-\!ax\!+\!b)(x^2\!+\!ax\!+\!b)\\
&=f_{T_k}(x)\\
&=xf_{T_k-u}(x)-(n_1f_{P_0}(x)\frac{f_{T_k-u}(x)}{f_{P_1}(x)}+\cdots+n_5f_{P_4}(x)\frac{f_{T_k-u}(x)}{f_{P_5}(x)}+f_{P_{k-1}}(x)\frac{f_{T_k-u}(x)}{f_{P_{k}}(x)})\\
&=\frac{f_{T_k-u}(x)}{c_k(x)}[xc_k(x)-(n_1f_{P_0}(x)\frac{c_k(x)}{f_{P_1}(x)}+\cdots+ n_5f_{P_4}(x)\frac{c_k(x)}{f_{P_5}(x)}+f_{P_{k-1}}(x)\frac{c_k(x)}{f_{P_{k}}(x)})]\\
&=\frac{f_{T_k-u}(x)}{c_k(x)}\cdot h_k(x),
\end{array}
\end{equation}
where $h_k(x)=xc_k(x)-(n_1f_{P_0}(x)\frac{c_k(x)}{f_{P_1}(x)}+\cdots + n_5f_{P_4}(x)\frac{c_k(x)}{f_{P_5}(x)}+f_{P_{k-1}}(x)\frac{c_k(x)}{f_{P_{k}}(x)})$ is a polynomial with degree $\partial(h_k(x))=\partial(c_k(x))+1$. It is clear that $(x^2-ax+b)(x^2+ax+b)|h_k(x)$. Let $h_k^*(x)=\frac{h_k(x)}{(x^2-ax+b)(x^2+ax+b)}$. We see that $h_k^*(x)$ contains only the basis factors in $S$ and   all such factors are collected in the set $S_k'$. Thus $h_k^*(x)=\prod_{p(x)\in S_k'}p(x)$  is specified as  $h_k^*(x)=h_{S_k'}(x)$. For  a basis factor  $s(x)\in S$, let $s_P=\{ P_i\mid s(x)|f_{P_i}(x), 1\le i\le 5\}$. For instance, let $s(x)=x^2-1,t(x)=x\in S$ then $s_P=\{ P_2,P_5\}$ and $t_P=\{P_1,P_3,P_5\}$.
\begin{lem}\label{starlike-P6-20}
Under the above assumptions, for $f_{T_k}(x)=\frac{f_{T_k-u}(x)}{c_k(x)}h_k(x)$ we have\\
(i) $\frac{f_{T_k-u}(x)}{c_k(x)}$ only contains basis factors in $S$.\\
(ii) $h_k(x)=h_{S_k'}(x)\cdot (x^2-ax+b)(x^2+ax+b)$, and we have $s(x)\parallel f_{T_k}(x)$ if $s(x)\in S_k'$.\\
(iii) For $s(x)\in S$, if $s(x)|\frac{f_{T_k-u}(x)}{c_k(x)}$ then $T_k-u$ contains at least two pathes in $s_P$ ( repeating is permitted  ) .\\
(iv) For $s(x)\in S$, if $s(x)|h_k(x)$ then  $P_i\not\in T_k-u$ for any  $P_i \in s_P$.\\
(v) $gcd(\frac{f_{T_k-u}(x)}{c_k(x)}, h_k(x))=1$, $gcd(c_k(x), h_k(x))=1$ and $gcd(c_{S_k}(x), h_{S_k'}(x))=1$.
\end{lem}
\begin{proof}
Note that $f_{T_k-u}(x)=f_{P_1}^{n_1}(x)f_{P_2}^{n_2}(x)f_{P_3}^{n_3}(x)f_{P_4}^{n_4}(x)f_{P_5}^{n_5}(x)\cdot f_{P_k}(x)$ and $q_k(x)\|f_{P_k}(x), q_k(x)\|c_k(x)$, $\frac{f_{T_k-u}(x)}{c_k(x)}$   only contains basis factors in $S$. Thus (i) follows.

From (\ref{P-eq-3}) and (i), we have $(x^2-ax+b)(x^2+ax+b)|h_k(x)$ and other factor $s(x)$ of $h_k(x)$ ( if any ) is basis factor in $S$. Moreover if $s^2(x)| f_{T_k}(x)$ then the root $\lambda$ of $s(x)$ is multiple and so $\lambda$ is a root of $f_{T_k-u}(x)$ by Lemma \ref{starlike-mul}. However, by the definition $h_k(x)$ does not contain any root of $f_{P_j}(x)$ for $P_j\in T_k-u$, it is a contradiction. Thus (ii) follows.

If $s(x)|\frac{f_{T_k-u}(x)}{c_k(x)}$ then $s^2(x)| f_{T_k-u}(x)$ and thus $T_k-u$ contains  two pathes ( repeating is permitted  ) whose characteristic polynomials have the factor  $s(x)$. (iii) follows.

On the contrary suppose that $P_i\in T_k-u$ such that  $s(x)| f_{P_i}(x)$. We claim that such $P_i$ must be unique in $s_P$ since otherwise $s(x)$ is a multiple factor of $f_{T_k-u}(x)$ and so $s(x)|\frac{f_{T_k-u}(x)}{c_k(x)}$, thus $s(x)^2|f_{T_k}(x)=\frac{f_{T_k-u}(x)}{c_k(x)}h_k(x)$, it contradicts (ii).  Notice that \begin{equation}\label{111}\begin{array}{ll}
h_k(x)&=xc_k(x)-(n_1f_{P_0}(x)\frac{c_k(x)}{f_{P_1}(x)}+\cdots + n_5f_{P_4}(x)\frac{c_k(x)}{f_{P_5}(x)}+f_{P_{k-1}}(x)\frac{c_k(x)}{f_{P_{k}}(x)})\\
&=s(x)u(x)-n_if_{P_{i-1}}(x)\frac{c_k(x)}{f_{P_i}(x)}.
\end{array}
\end{equation}
From (\ref{111}) we deduce that $s(x)\nmid h_k(x)$. It is a contradiction. (iv) follows.

(v) follows immediately from (iv).
\end{proof}

For $6\le k\le 13$, we have $c_k(x)=c_{S_k}(x)\cdot q_k(x)$ and $h_k(x)=h_{S_k'}(x)\cdot (x^2-ax+b)(x^2+ax+b)$ defined above, where $S_k$ and $S_k'$ are disjoint subset of $S$ according to Lemma \ref{starlike-P6-20} (v).
Now we refer to $S_k\cup S_k'$ as a pair of feasible  factor set ( with respect to $T_k-u=n_1P_1\cup n_2 P_2\cup\cdots\cup n_5P_5\cup P_k$ ).
\begin{lem}\label{S-lem-1}
For a pair of feasible factor set $S_k\cup S_k'$ we have\\
(i) $S_k\cup S_k'\subseteq S$ and $gcd(s_k(x),s_k'(x))=1$ for any $s_k(x)\in S_k$ and $s_k'(x)\in S_k'$.\\
(ii) $\partial(h_{S_k'}(x))-\partial(c_{S_k}(x))=\partial(q_k(x))-3$.\\
(iii) $1\le\partial(c_{S_k}(x))\le 7-\frac{\partial(q_k(x))}{2}$ and  $\partial(h_{S_k'}(x))\le 4+\frac{\partial(q_k(x))}{2}$.
\end{lem}
\begin{proof}
(i) follows from  Lemma \ref{starlike-P6-20}(v).

From definition, we have
$1=\partial(h_k(x))-\partial(c_k(x))=(\partial(h_{S_k'}(x))+4)-(\partial(c_{S_k}(x))+\partial(q_k(x))$, which leads to (ii).

From (i) we have $\partial(h_{S_k'}(x))+\partial(c_{S_k}(x))\leq \partial(\prod_{s(x)\in S}s(x))=11$, and combining with (ii) we get (iii).
\end{proof}

\begin{lem}\label{starlike-P6-22}
Let $T \supset K_{1,3} $ be a quadratic starlike tree with the center vertex $u$. If $f_{T}(x)$ is of  form (II) then $T-u$ can not  contain  any path $P_k$ with   $k>5$.
\end{lem}
\begin{proof}
Suppose to  contrary that    we may assume that $T-u=n_1P_1\cup n_2 P_2\cup\cdots\cup n_5P_5\cup P_k$ where  $n_1+n_2+n_3+n_4+n_5> 1$ and  $6\le k\le 13$. We now identify $T$ as  $T_k$ to relate with $P_k$.
According to Corollary \ref{starlike-the-third-largest-1}, we have $\Lambda(T_k) = \{0,\pm1,\pm\sqrt{2}, \pm\frac{1\pm\sqrt{5}}{2},\pm \sqrt{3},\pm\frac{a\pm \sqrt{a^2-4b}}{2}\}$ where $\lambda_1(T_k)=\frac{a+\sqrt{a^2-4b}}{2}$. It is clear that $\lambda_1(T_k-u)=\lambda_1(P_k)$ and $\lambda_2(T_k-u)\ge \lambda_2(P_k)$.

First of all, according to the  proof of   Lemma \ref{starlike-P6} our conclusion holds if $\lambda_2(T_k)\le \sqrt{3}$. In what follows we always assume that $\lambda_2(T_k)> \sqrt{3}$. It implies that $\lambda_2(T_k)=|\bar{\lambda}_1(T_k)|$ where $\bar{\lambda}_1(T_k)$ is  conjugate with $\lambda_1(T_k)$.
Notice that the eigenvalues of $T_k-u$ are included in $\Lambda(T_k-u)=\{0,\pm1,\pm\sqrt{2}, \pm\frac{1\pm\sqrt{5}}{2},\pm \sqrt{3}\}\cup\{\lambda_i(P_k)\mid i=1,2,...,k\}$. We can respectively ordered these eigenvalues of $T_k-u$ for $k=6,7,...,13$.
\begin{equation}\label{O-eq-1}
\scriptsize
\begin{array}{ll}
T_6-u& 0<\lambda_3(P_6)<\frac{-1+\sqrt{5}}{2}<1<\lambda_2(P_6)<\sqrt{2}<\frac{1+\sqrt{5}}{2}<\sqrt{3}<\lambda_1(P_6)<2 \\
T_7-u& 0<\frac{-1+\sqrt{5}}{2}<\lambda_3(P_7)<1<\sqrt{2}=\lambda_2(P_7)<\frac{1+\sqrt{5}}{2}<\sqrt{3}<\lambda_1(P_7)<2\\
T_8-u&0<\lambda_4(P_8)<\frac{-1+\sqrt{5}}{2}<1=\lambda_3(P_8)<\sqrt{2}<\lambda_2(P_8)<\frac{1+\sqrt{5}}{2}<\sqrt{3}<\lambda_1(P_8)<2 \\
T_9-u&0<\frac{-1+\sqrt{5}}{2}=\lambda_4(P_9)<1<\lambda_3(P_9)<\sqrt{2}<\frac{1+\sqrt{5}}{2}=\lambda_2(P_9)<\sqrt{3}<\lambda_1(P_9)<2\\
T_{10}-u&0<\lambda_5(P_{10})<\frac{-1+\sqrt{5}}{2}<\lambda_4(P_{10})<1<\lambda_3(P_{10})<\sqrt{2}<\frac{1+\sqrt{5}}{2}<\lambda_2(P_{10})<\sqrt{3}<\lambda_1(P_{10})<2\\
T_{11}-u&0<\lambda_5(P_{11})<\frac{-1+\sqrt{5}}{2}<1=\lambda_4(P_{11})<\sqrt{2}=\lambda_3(P_{11})<\frac{1+\sqrt{5}}{2}<\sqrt{3}=\lambda_2(P_{11})<\lambda_1(P_{11})<2\\
T_{12}-u& 0<\lambda_6(P_{12})<\frac{-1+\sqrt{5}}{2}<\lambda_5(P_{12})<1<\lambda_4(P_{12})<\sqrt{2}<\lambda_3(P_{12})<\frac{1+\sqrt{5}}{2}<\lambda_2(P_{12})<\sqrt{3}<\lambda_1(P_{12})<2\\
T_{13}-u&0<\lambda_6(P_{13})<\frac{-1+\sqrt{5}}{2}<\lambda_5(P_{13})<1<\lambda_4(P_{13})<\sqrt{2}<\lambda_3(P_{13})<\frac{1+\sqrt{5}}{2}<\sqrt{3}<\lambda_2(P_{13})<\lambda_1(P_{13})<2\\
\end{array}\end{equation}

{\flushleft\bf Case 1.} First we suppose that  $\sqrt{3}\in Spec(T_k)$.

By Interlacing Theorem, from (\ref{O-eq-1}) we see that $\lambda_3(T_k)=\sqrt{3}<\lambda_2(T_k)=|\bar{\lambda}_1(T_k)|< \lambda_{1}(T_k-u)=\lambda_1(P_k)<\lambda_1(T_k)$. We consider the following two subcases bellow.

{\bf Subcase 1.1 } Assume that  $\sqrt{3}\in Spec(T_k-u)$. Then $s(x)=x^2-3$ is a factor of $c_k(x)$, and thus $x^2-3\nmid h_{S_k'}(x)$ by Lemma \ref{starlike-P6-20}(v). It means that $x^2-3\mid \frac{f_{T_k-u}(x)}{c_k(x)}$ due to $\sqrt{3}\in Spec(T_k)$ and so  $|s_P|\geq2$ by Lemma \ref{starlike-P6-20}(iii). Furthermore, from (\ref{f-eq-1}) we see that   $P_5 \in s_P$ because $s_P$ contains $P_{11}$ at most one times, and so $n_5\geq1$.  Thus  $\{x,x^2-1,x^2-3\}\subseteq S_k$ and so $\partial(c_{S_k}(x))\ge 5$. According to Lemma \ref{S-lem-1}(iii), $5\le\partial(c_{S_k}(x))\le 7-\frac{\partial(q_k(x))}{2}$ and thus $\partial(q_k(x))\le 4$, which, from (\ref{f-eq-1}), leads to   $k=7, 9$ or $11$( in fact, $\partial(q_{7}(x))=\partial(q_{9}(x))=\partial(q_{11}(x))=4$ ). On the other hand, we see from (\ref{f-eq-1}) that $\{x,x^2-1,x^2-3\}\subsetneq S_k$ for $k=7, 9, 11$ and  so $5<\partial(c_{S_k}(x))\le 7-\frac{\partial(q_k(x))}{2}$. Thus $\partial(q_{7}(x)),\partial(q_{9}(x)),\partial(q_{11}(x))< 4$. It is a contradiction.

{\bf Subcase 1.2 } Assume that $\sqrt{3}\not\in Spec(T_k-u)$. From (\ref{O-eq-1}), we see that the interval $I=[\lambda_3(T_k)=\sqrt{3}, \lambda_2(T_k)=|\bar{\lambda}_1(T_k)|]$ does not contain any eigenvalue of $T_k-u$ if  $k\not=13$ (notice that $\lambda_2(P_{13})\in I$),  which contradicts  Interlacing Theorem. As a supplement, we need to verify that $k\not=13$. Otherwise, let $k=13$. Then $1\leq \partial(c_{S_{13}}(x))\le 7-\frac{\partial(q_{13}(x))}{2}=1$ by Lemma \ref{S-lem-1}(iii), and thus $S_{13}=\{x\}$. It implies that $T_{13}-u=n_1P_1\bigcup P_{13}$ where $n_1\ge 2$. On the other hand, we have $\partial(h_{S_{13}'}(x))=\partial(c_{S_{13}}(x))+9=10$ by Lemma \ref{S-lem-1}(ii). Thus $S_{13}'=\{x^2-1,x^2-2,x^2-3,(x^2-x-1)(x^2+x-1)\}$ by Lemma \ref{starlike-P6-20}(v), and then $T_{13}$ has eigenvalue $\frac{1+\sqrt{5}}{2}$ by Lemma \ref{starlike-P6-20}(ii). However, from (\ref{O-eq-1}) we see that the interval  $[\frac{1+\sqrt{5}}{2},\sqrt{3}]$ has no any eigenvalue of $T_{13}-u=n_1P_1\bigcup P_{13}$ and its two  endpoints  are eigenvalues of $T_{13}$, which   contradicts Interlacing Theorem.

{\flushleft\bf Case 2.} Next we suppose that  $\sqrt{3}\not\in Spec(T_k)$.

Since $f_{T_k}(x)=\frac{f_{T_k-u}(x)}{c_k(x)}\cdot h_k(x)$, $\sqrt{3}$ is neither a root of $\frac{f_{T_k-u}(x)}{c_k(x)}$ nor of $h_k(x)$.  Thus $x^2-3 \not\in S_k'$, and by Lemma \ref{starlike-P6-20}(iii) at most  one of $P_5$ and $P_{11}$  contains in $T_k-u$.

{\bf Subcase 2.1 } Assume that $\sqrt{3}\in Spec(T_k-u)$.
Then there is exactly one of $P_5$ and $P_{11}$ in $T_k-u$.   First, let  $P_{11}\in T_k-u$. Note that $f_{11}(x)=x(x^2 - 1)(x^2 - 2)(x^2 - 3)(x^4 -4x^2+1)$, we have $\{x, x^2-1,x^2-2,x^2-3\}\subseteq S_{11}$. By Lemma \ref{S-lem-1}(i), $S_{11}'=\{(x^2-x-1)(x^2+x-1)\}$ or $S_{11}'=\emptyset$. By Lemma \ref{starlike-P6-20}(ii), we have $1=\partial(q_k(x))-3=\partial(h_{S_{11}'}(x))-\partial(c_{S_{11}}(x))\le4-7=-3$,  a contradiction.  Next, let  $P_{5}\in T_k-u$. Then $\{x,x^2-1,x^2-3\}\subseteq S_k$, which leads to a contradiction as  in Subcase 1.1.

{\bf Subcase 2.2 } Assume that $\sqrt{3}\not\in Spec(T_k-u)$.
Then $P_5,P_{11} \not\in T_k-u$, and so   $x^2-3\notin S_k$. Therefore, the feasible  factor set $S_k\cup S_k'\subseteq \{x, x^2-1, x^2-2, (x^2-x-1)(x^2+x-1)\}=S\setminus\{x^2-3\}$, and then $\partial(h_{S_k'}(x))+\partial(c_{S_k}(x))\leq9$. By Lemma \ref{S-lem-1}(ii),  $1\le\partial(c_{S_k}(x))\leq 6-\frac{\partial(q_k(x))}{2}$. We will consider the following situations.

First let $\frac{1+\sqrt{5}}{2}\in Spec(T_k)$ and $\frac{1+\sqrt{5}}{2}\in Spec(T_k-u)$. Clearly,  $s(x)=(x^2-x-1)(x^2+x-1)\in S$ has a root $\frac{1+\sqrt{5}}{2}$. By Lemma \ref{starlike-P6-20}(iii) and (iv), we have $\{P_4,P_9\}\subseteq s_P$ or $\{2*P_4\} \subseteq s_P$, and so $n_4\geq1$.  Then $\{(x^2-x-1)(x^2+x-1)\}\subseteq S_k$ and so  $4\le\partial(c_{S_k}(x))\le 6-\frac{\partial(q_k(x))}{2}$ and thus $\partial(q_k(x))\le 4$, which leads to   $k=7,9$ from (\ref{f-eq-1}) and indeed, $\partial(q_{7}(x))=\partial(q_{9}(x))=4$. On the other hand, we see from (\ref{f-eq-1}) that $\{(x^2-x-1)(x^2+x-1)\} \subsetneq S_7,S_9$ and  so $4<\partial(c_{S_7}(x)),\partial(q_{9}(x))$, a contradiction.

Next let $\frac{1+\sqrt{5}}{2}\in Spec(T_k)$ and $\frac{1+\sqrt{5}}{2}\not\in Spec(T_k-u)$. Then $P_4, P_9\notin T_k-u$ and so $s(x)=(x^2-x-1)(x^2+x-1)\notin S_k$. However, since $\frac{1+\sqrt{5}}{2}\in Spec(T_k)$, we have $s(x)\in S_k'$. From (\ref{O-eq-1}), we see that the interval $I=[\lambda_3(T_k), \lambda_2(T_k)]$ does not contain any eigenvalue of $T_k-u$ if  $k\not=10,12,13$. Thus $k=10,12$ or $13$ by  Interlacing Theorem ( notice that $\lambda_2(P_{10})$,  $\lambda_2(P_{12})$, $\lambda_2(P_{13})\in I$ ). However,  note that $\partial(q_{12}(x))=\partial(q_{13}(x))=12$, we have $1\le\partial(c_{S_k}(x))\leq 6-\frac{\partial(q_k(x))}{2}=0$ for $k=12,13$, a contradiction. It remains to consider $k=10$. It is clear that $1\le\partial(c_{S_{10}}(x))\leq 6-\frac{\partial(q_{10}(x))}{2}=1$ due to $\partial(q_{10}(x))=10$. Thus $S_{10}=\{x\}$ and $T_{10}-u=n_1P_1\cup P_{10}$ where $n_1\geq2$. By Lemma \ref{S-lem-1}(ii), we get that $\partial(h_{S_{10}'}(x))=8$.  From Lemma \ref{S-lem-1}(i), $S_{10}'=\{x^2-1,x^2-2, (x^2-x-1)(x^2+x-1)\}$. Then $T_{10}$ has eigenvalues $\sqrt{2}$ and  $\frac{1+\sqrt{5}}{2}$. However, from (\ref{O-eq-1}) we see that the interval  $[\sqrt{2}, \frac{1+\sqrt{5}}{2}]$ has no any eigenvalue of $T_{10}-u=n_1P_1\bigcup P_{10}$, which   contradicts  Interlacing Theorem.

Third let $\frac{1+\sqrt{5}}{2}\notin Spec(T_k)$ and $\frac{1+\sqrt{5}}{2}\in Spec(T_k-u)$. Then there is exactly one of $P_4$ and $P_9$ in $T_k-u$.   Note that $f_9(x)=x(x^2-x-1)(x^2+x-1)(x^4 -5x^2+5)$, if $P_9\in T_k-u$ then $\{x, (x^2-x-1)(x^2+x-1)\}\subseteq S_9$. By Lemma \ref{S-lem-1}(i), we have $S_9'=\{x^2-1,x^2-2\}$, $\{x^2-1\}$, $\{x^2-2\}$ or $\emptyset$.
 By Lemma \ref{S-lem-1}(ii), we have $1=\partial(q_9(x))-3=\partial(h_{S_9'}(x))-\partial(c_{S_9}(x))\le4-5=-1$,  a contradiction. If  $P_4\in T_k-u$ then $\{(x^2-x-1)(x^2+x-1)\}\subseteq S_k$, which leads to a contradiction as  in the first situation.

 Fourth let  $\frac{1+\sqrt{5}}{2}\notin Spec(T_k)$ and $\frac{1+\sqrt{5}}{2}\notin Spec(T_k-u)$. We have $P_4, P_9, P_{11}\notin T_k-u$ and then $s(x)=(x^2-x-1)(x^2+x-1)\notin S_k,S_k'$. Then $S_k \cup S_k'\subseteq \{x, x^2-1, x^2-2\}$ and so $\partial(h_{S_k'}(x))+\partial(c_{S_k}(x))\leq 5$. Combining with Lemma \ref{S-lem-1}(ii), we have $1\le\partial(c_{S_k}(x))\leq 4-\frac{\partial(q_k(x))}{2}$, which gives that $k\not=10,12,13$ due to $\frac{\partial(q_k(x))}{2}>3$ for $k=10,12,13$. It remains to consider the cases of $k=6,7,8$.
If $k=8$ then $1\le\partial(c_{S_8}(x))\leq 4-\frac{\partial(q_8(x))}{2}=1$. However, $\{x^2-1\}\subseteq S_8$, a contradiction.
If $k=7$ then $1\le\partial(c_{S_7}(x))\leq 4-\frac{\partial(q_7(x))}{2}=2$. However, $\{ x, x^2-2\}\subseteq S_7$ from (\ref{f-eq-1}) and so $\partial(c_{S_7}(x))\geq3$, a contradiction.
If $k=6$ then $1\le\partial(c_{S_6}(x))\leq 4-\frac{\partial(q_6(x))}{2}=1$ and   thus $S_6=\{x\}$ and $T_6-u=n_1P_1\cup P_6$ where $n_1\geq2$. On the other aspect, we get that $\partial(h_{S_6'}(x))=4$  from Lemma \ref{S-lem-1}(ii) and then $S_6'=\{x^2-1,x^2-2\}$ by Lemma \ref{S-lem-1}(i). Thus $T_6$ has an eigenvalue $\sqrt{2}$  by Lemma \ref{starlike-P6-20}(ii). However, from (\ref{O-eq-1}) we see that the interval  $[\sqrt{2}, \lambda_2(T_6)]$ has no any eigenvalue of $T_6-u=n_1P_1\bigcup P_6$, which   contradicts  Interlacing Theorem.

We complete this proof.
\end{proof}

\begin{thm}\label{thm-quadratic-starlike-II}
Let $T \supset K_{1,3} $ be a quadratic starlike tree. Then $f_{T}(x)$ is of  form (II) if and only if  $T$ is one of the five graphs  listed in Table \ref{table-quadratic-II}.
\end{thm}
\begin{proof}
According to Lemma \ref{starlike-P6-22}, we have  $T-u=n_1P_1\cup n_2P_2\cup n_3P_3\cup n_4P_4\cup n_5P_5$ where $n_1+n_2+\cdots+n_5\geq3$.
Moreover, we have the character equation from Lemma \ref{lem-z}
\begin{equation}\label{h-eq-2}
\begin{array}{ll}
&t_{(n_1,n_2,n_3,n_4,n_5)}(x)\\&\!\!\!=xm(x)-(n_1(x^2-1)(x^2-2)(x^2-x-1)(x^2+x-1)(x^2 - 3)\\
&\!\!\!+n_2x^2(x^2\!-\!2)(x^2\!-\!x\!-\!1)(x^2\!+\!x\!-\!1)(x^2\! - \! 3)+n_3(x^2\!-\!1)^2(x^2\!-\!x\!-\!1)(x^2\!+\!x\!-\!1)(x^2\! - \!3)\\
&\!\!\!+n_4x^2(x^2\!-\!1)(x^2\!-\!2)^2(x^2\! - \!3)+n_5(x^2\!-\!2)(x^2\!-\!x\!-\!1)^2(x^2\!+\!x\!-\!1)^2)\\
&\!\!\!=x^{z_1}(x^2\!-\!1)^{z_2}(x^2\!-\!2)^{z_3}((x^2\!-\!x\!-\!1)(x^2\!+\!x\!-\!1))^{z_4}(x^2\!-\!3)^{z_5}(x^2\!\!-\!ax\!+\!b)(x^2\!+\!ax\!+\!b)\\
&\!\!\!=u_{(z_1,z_2,z_3,z_4,z_5)}(x),
\end{array}
\end{equation}
and the  parameter equation
\begin{equation}\label{z-eq-1}
z_1+2z_2+2z_3+4z_4+2z_5+4=12.
\end{equation}
Using the restriction condition (\ref{eq-z}), it is routine to find all the  thirteen solutions $(z_1,z_2,z_3,z_4,z_5)$ of  (\ref{z-eq-1}) that are listed in Table \ref{table-quadratic-II-invalid} and \ref{table-quadratic-II}, respectively. Now we divide  our proof in two steps.

{\flushleft\bf Step 1. } All the solutions $(z_1,z_2,z_3,z_4,z_5)$ listed   in  Table \ref{table-quadratic-II-invalid} are invalid.

By taking $(z_1,z_2,z_3,z_4,z_5)=(0,2,1,0,1)$ in the first row of Table \ref{table-quadratic-II-invalid}, we get that $n_2=n_3=n_5=0$ and $n_1,n_4\geq1$ from the restriction (\ref{eq-z}). Thus the Eq. (\ref{h-eq-2}) becomes
 $$\begin{array}{ll}
&xm(x)-(n_1(x^2-1)(x^2-2)(x^2\!-\!x\!-\!1)(x^2\!+\!x\!-\!1)(x^2-3)
+n_4x^2(x^2\!-\!1)(x^2\!-\!2)^2(x^2\!-\!3))\\
&=t_{(n_1,0,0,n_4,0)}(x)\\
&=u_{(0,2,1,0,1)}(x)\\
&=(x^2-1)^{2}(x^2-2)(x^2-3)(x^2-ax+b)(x^2+ax+b).
\end{array}$$
By   deleting  $(x^2-1)(x^2-2)(x^2-3)$ on two sides, the character equation is simplified as
$$x^2(x^2\!-\!x\!-\!1)(x^2\!+\!x\!-\!1)-(n_1(x^2\!-\!x\!-\!1)(x^2\!+\!x\!-\!1)
+n_4x^2(x^2\!-\!2))
=(x^2-1)(x^2-ax+b)(x^2+ax+b).$$
By comparing the  coefficients, we get the restriction condition
$$\left\{\begin{array}{ll}
n_1+n_4+3=a^2-2b+1\\
3n_1+2n_4+1=a^2+b^2-2b\\
n_1=b^2\\
\end{array}\right.
\Longrightarrow
\left\{\begin{array}{ll}
n_1=b^2\geq1\\
n_4=-b^2+1\geq1\\
a^2=2b+3\\
\end{array}\right.$$
Note that $n_4=-b^2+1\geq1$, this is impossible since $\Delta=a^2-4b$ is  a square-free number.  Thus such a solution is invalid.

As the same process as above, from the solutions of $(z_1,z_2,z_3,z_4,z_5)$  in 2th column  of Table \ref{table-quadratic-II-invalid},  we can respectively determine $(n_1,n_2,n_3,n_4,n_5)$, according to (\ref{eq-z}), in 3th  column  of Table \ref{table-quadratic-II-invalid}, which returns respectively to  Eq. (\ref{h-eq-2}) we get the simplification of character equations, by comparing the  coefficients we get respectively the restriction conditions in 4th  column  of Table \ref{table-quadratic-II-invalid}. All these restriction conditions are conflicting, which are indicated in detail in 5th  column  of Table \ref{table-quadratic-II-invalid}.

\begin{table}[H]
\footnotesize
\caption{\small Invalid solutions  of form (II)}
\centering
\scriptsize
\renewcommand\arraystretch{1.1}
\begin{tabular*}{16cm}{m{4pt}|m{45pt}|m{50pt}|m{155pt}|m{175pt}}
\hline
&$(\!z_1,z_2,z_3,z_4,z_5\!)$ &$(\!n_1,n_2,n_3,n_4,n_5\!)$& restriction condition& simplification of restriction condition\\\hline
\multirow{2}*{1} &\multirow{2}*{$(0,2,1,0,1)$} & \multirow{2}*{$( n_1,0,0,n_4,0)$ }& $n_1+n_4+3=a^2-2b+1$&\multirow{2}*{$n_1\!=\!b^2\!\geq\!1$, $n_4\!=\!-b^2+1\geq1$ }\\
&&&$3n_1+2n_4+1=a^2+b^2-2b$, $n_1=b^2$&\\\hline
\multirow{3}*{2} &\multirow{3}*{$(0,2,0,0,2)$} & \multirow{3}*{$(n_1, 0,n_3,n_4,0)$ } & $n_1+n_3+n_4+5=a^2-2b+4$,&\multirow{3}*{$n_1\!=\!b^2\!\geq\!0$, $n_3\!=\!b^2\!\geq\!1$,  $n_4\!=\!-b^2\!+\!1\!\geq\!1$ }\\
&&&$5n_1 + 4n_3 + 4n_4 + 7=4a^2 + b^2 - 8b + 3$ &\\ &&&$7n_1+4n_3+4n_4+2=3a^2+4b^2- 6b$, $2n_1 + n_3= 3b^2$&\\\hline
\multirow{2}*{3} &\multirow{2}*{$(0,1,1,0,2)$} & \multirow{2}*{$( n_1, 0,0,n_4,0)$ }  & $n_1+n_4+3= a^2- 2b+ 3$& \multirow{2}*{$n_1=3b^2\geq1$, $n_4=-b^2+1\geq1$ } \\
&&&$3n_1 + 2n_4 + 1=3a^2 + b^2 - 6b$, $n_1=3b^2$&\\\hline
\multirow{3}*{4} &\multirow{3}*{$(0,0,2,0,2)$} & \multirow{3}*{$( n_1,n_2,0,n_4,0)$ }&$n_1+ n_2+n_4+4=a^2-2b+5$&\multirow{2}*{$ n_1=6b^2\geq1$, $n_2=-3b^2+1\geq1$,}\\
&&&$4n_1 + 3n_2 + 3n_4 + 4=5a^2 + b^2 - 10b + 6$&\\&&& $4n_1+n_2+2n_4+1= 6a^2+5b^2-12b$, $n_1=6b^2$
&$n_4=-a^2+b^2+(b+1)^2\geq1$  \\\hline
\multirow{2}*{5} &\multirow{2}*{$(2,0,1,0,2)$ }& \multirow{2}*{$(0,n_2,0,n_4,0)$ } &$n_2+n_4+4= a^2-2b +3$,& $n_2\!=\!2a^2\!\geq\!1$, $n_4\!=\!-a^2\!+\!1\!\geq\!1$ \mbox{ or } \\ &&& $3n_2 + 3n_4 + 4=3a^2 + b^2 - 6b$, $n_2+2n_4+1=3b^2$&$n_2\!=\!-a^2\!+\!5\!\geq\!1$, $n_4\!=\!2a^2\!-\!8\!\geq\!1$  \\\hline
\multirow{2}*{6 } &\multirow{2}*{$(0,0,2,1,0)$} & \multirow{2}*{$(n_1,n_2,0,0,n_5)$} &$n_1+n_2+n_5+4=a^2-2b +2$
 &$n_1\!=\!-a^2\!+\!(b+1)^2\!+\!2\!\geq\!0$, $n_2\!=\!-a^2\!+\!2b\!+\!4\!\geq\!0$, \\
 &&&$4n_1 + 3n_2 + 3n_5 + 3=2a^2 + b^2 - 4b$, $3n_1+n_5=2b^2$& $n_5\!=\!3a^2\!-\!(b+3)^2\!\geq\!1$, $\Rightarrow$ $3n_1+n_5=2b^2=1$\\
 \hline
\multirow{2}*{7} &\multirow{2}*{$(0,0,0,1,2)$} & \multirow{2}*{$(n_1,n_2,n_3,0,0)$}& $n_1+n_2+n_3+3=a^2 - 2b+ 3$ &$n_1\!=\!a^2\!+\!(b-1)^2\!-\!3\!\geq\!0$, $n_2\!=\!2a^2\!-\!2(b+1)^2\!\geq\!1$,\\&&&$3n_1 + 2n_2 + 2n_3 + 2=3a^2 + b^2 - 6b$, $ 2n_1+n_3= 3b^2$ & $n_3\!=\!-2a^2\!+\!(b+2)^2\!\geq\!1$, $\Rightarrow$ $n_2\!+\!n_3\!=\!-b^2\!+\!2\!\geq\!2$  \\\hline
\multirow{2}*{8} &\multirow{2}*{$(2,0,2,0,1)$} & \multirow{2}*{$(0,n_2,0,n_4,0)$}  & $n_2+n_4+4= a^2- 2b + 2$& $n_2=1$, $n_4=b^2-1\geq1$, $a^2=(b+1)^2+1$ \\
&&&$3n_2 + 3n_4 + 4=2a^2 + b^2 - 4b$, $n_2+2n_4+1=2b^2$&$\Rightarrow b=-1, a=1$, $n_4=0$\\\hline
\end{tabular*}\label{table-quadratic-II-invalid}
\end{table}

{\flushleft\bf Step 2. } All the solutions $(z_1,z_2,z_3,z_4,z_5)$  listed  in  Table \ref{table-quadratic-II} are valid.

By taking $(z_1,z_2,z_3,z_4,z_5)=(0,0,0,2,0)$  in 1th row of Table \ref{table-quadratic-II}, we get $n_4=0$ from (\ref{eq-z}). Thus  Eq. (\ref{h-eq-2}) becomes
 $$\begin{array}{ll}
&\!\!\!\!\!\!xm(x)\!-\!(n_1(x^2\!-\!1)(x^2\!-\!2)(x^2\!-\!x\!-\!1)(x^2\!+\!x\!-\!1)(x^2\!-\!3)
\!+\!n_2x^2(x^2\!-\!2)(x^2\!-\!x\!-\!1)(x^2\!+\!x\!-\!1)(x^2\!-\!3)\\
&+n_3(x^2-1)^2(x^2-x-1)(x^2+x-1)(x^2-3)
+n_5(x^2-2)(x^2-x-1)^2(x^2+x-1)^2)\\&\!\!\!=t_{(n_1,n_2,n_3,0,n_5)}(x)\\&\!\!\!=u_{(0,0,0,2,0)}(x)\\
&\!\!\!=((x^2-x-1)(x^2+x-1))^2(x^2-ax+b)(x^2+ax+b).\\
\end{array}$$
By  deleting  $(x^2-x-1)(x^2+x-1)$ on two sides, we have
 $$\begin{array}{ll}&x^2(x^2-1)(x^2-2)(x^2-3)-(n_1(x^2-1)(x^2-2)(x^2-3)
+n_2x^2(x^2-2)(x^2-3)\\
&+n_3(x^2-1)^2(x^2-3)+n_5(x^2-2)\!(x^2-x-1)(x^2+x-1))\\&=(x^2-x-1)(x^2+x-1)(x^2-ax+b)(x^2+ax+b).\end{array} $$
By comparing the  coefficients of  the character equation, we get the restriction condition
\begin{equation}\label{eq-1235}
\!\!\!\!\left\{\!\!\begin{array}{ll}
n_1+n_2+n_3+n_5+6=a^2-2b+3\\
6n_1+5n_2+5n_3+5n_5+11=3a^2+b^2-6b+1\\
11n_1+6n_2+7n_3+7n_5+6=a^2+3b^2-2b\\
6n_1+3n_3+2n_5=b^2\\
\end{array}\right.
\!\!\!\!\Longrightarrow\!\!\!
\left\{\!\!\begin{array}{ll}
n_1=n_2=-2a^2+b^2+4b+5\\
n_3=2a^2-b^2-4b-4\geq1\\
n_5=3a^2-b^2-6b-9\geq1\\
\end{array}\right.\end{equation}
which leads to $n_1+n_3=1$, thus  $n_1=n_2=0$ and $n_3=1$. Furthermore,  from Eq. (\ref{eq-1235}), we deduce that $2a^2=(b+2)^2+1$ and $n_5=\frac{b^2-3}{2}\ge 2$ due to $d_T(u)=1+n_5\geq3$.
 Thus $T=T_{0,0,1,0,n_5}$ is a quadratic starlike tree with the restriction of   $n_5=\frac{b^2-3}{2}\geq2$ and  $2a^2=(b+2)^2+1$,  its characteristic polynomial is
 $$\begin{array}{ll}
f_T(x)&=u_{(0,0,0,2,0)}(x)\frac{f_{T-u}(x)}{m(x)}\\
&=((x^2-x-1)(x^2+x-1))^2(x^2-ax+b)(x^2+ax+b)\frac{(x^2-3)^{n_5}(x^2-2)(x^2-1)^{n_5}x^{n_5+1}}{m(x)}\\
&=x^{n_5}(x^2-1)^{n_5-1}(x^2-x-1)(x^2+x-1)(x^2-3)^{n_5-1}(x^2-ax+b)(x^2+ax+b).\\
\end{array}$$
As the same process  as above,  we can determine the other four quadratic starlike trees: $T_{0,0,0,n_4}$, $T_{2,0,0,n_4}$, $T_{n_1,0,n_3} $ and $T_{n_1,n_2}$, which are listed along with their restriction conditions in 4-column and the characteristic polynomials in $5$-column  of the  Table \ref{table-quadratic-II}, respectively.
 It is clear that the above five families of starlike trees are all quadratic with form (II). In addition, as we add to the  Remark \ref{re-1}, each of these five starlike trees produces an infinite class.

 We complete this proof.
\end{proof}

\begin{table}
\footnotesize
\caption{\small Quadratic starlike trees  of form (II)}
\centering
\scriptsize
\renewcommand\arraystretch{1.1}
\begin{tabular*}{15.5cm}{m{4pt}|m{45pt}|m{50pt}|m{165pt}|m{140pt}}
\hline
&$(\!z_1,z_2,z_3,z_4,z_5\!)$ &$(\!n_1,n_2,n_3,n_4,n_5\!)$& $T$/restriction conditions&$f_T(x)$\\\hline
\multirow{2}*{1 }& \multirow{2}*{$(0,0,0,2,0)$ }& \multirow{2}*{$(0 ,0,1,0,n_5)$}& $T_{0,0,1,0,n_5}$&$x^{n_5}(x^2-1)^{n_5-1}(x^2-x-1)(x^2+x-1)$\\
&&&$n_5=\frac{b^2-3}{2}\geq2, 2a^2=(b+2)^2+1$ &$\cdot(x^2-3)^{n_5-1}(x^2-ax+b)(x^2+ax+b)$
\\\hline
\multirow{2}*{2} &\multirow{2}*{$(2,1,1,0,1)$} & \multirow{2}*{$(0,0,0, n_4,0)$}  & $T_{0,0,0,n_4}$&$x((x^2-x-1)(x^2+x-1))^{n_4-1}$\\
&&&$n_4=\frac{b^2-1}{2}\geq3,2a^2=(b+2)^2+1 $& $\cdot(x^2-ax+b)(x^2+ax+b)$\\\hline
\multirow{4}*{ 3}&\multirow{4}*{$(0,1,2,0,1)$ }& \multirow{4}*{$(2,0,0, n_4,0)$ } &$T_{2,0,0,n_4}$&$x(x^2-2)((x^2-x-1)(x^2+x-1))^{n_4-1}$\\
&&&$n_4=a^2-5\geq1, b=1 \mbox{  or }$&$\cdot(x^2-ax+1)(x^2+ax+1) \ \mbox{or}$\\
&&&\multirow{2}*{$n_4=a^2-1\geq1,b=-1 $}&$x(x^2-2)((x^2-x-1)(x^2+x-1))^{n_4-1}$\\
&&&&$\cdot(x^2-ax-1)(x^2+ax-1)$ \\\hline
\multirow{2}*{4} &\multirow{2}*{$(0,1,0,1,1)$} & \multirow{2}*{$(n_1,0,n_3,0,0)$}&$T_{n_1,0,n_3}$&$x^{n_1+n_3-1}(x^2-2)^{n_3-1}$\\ &&&$n_1\!=\!-a^2\!+\!(b\!+\!1)^2\!+\!1 ,n_3\!=\!2a^2\!-\!(b+2)^2\!\geq\!1,n_1+n_3\geq3$&
$\cdot(x^2-ax+b)(x^2+ax+b)$\\\hline
\multirow{2}*{5} &\multirow{2}*{$(0,0,1,1,1)$} &\multirow{2}*{$(n_1, n_2,0,0,0)$} & $T_{n_1,n_2}$&$x^{n_1-1}(x^2-1)^{n_2-1}$\\
&&&$n_1=b^2\geq1,n_2=a^2-(b+1)^2\geq1,n_1+n_2\geq3$&$\cdot(x^2-ax+b)(x^2+ax+b)$\\\hline
\end{tabular*}\label{table-quadratic-II}
\end{table}

\begin{remark}\label{re-con}
According to Theorem \ref{star-thm-1}, we know that if $f_T(x)$ is of from (II), then $(x^2-ax+b)$ and $(x^2+ax+b)$ must be    irreducible. However our restriction conditions for quadratic starlike trees in Table  \ref{table-quadratic-II} are not sufficient to guarantee the reducibility of $(x^2-ax+b)$ and $(x^2+ax+b)$. Thus the condition that $\Delta=a^2-4b$ is  a square-free number is also necessary for the restriction of quadratic starlike trees in Table  \ref{table-quadratic-II}.  As a supplement we give an example for $T_{n_1,0,n_3}$.

$T_{n_1,0,n_3}$ has the restriction conditions $n_1=-a^2+(b+1)^2+1$, $n_3=2a^2-(b+2)^2\geq1$ and $n_1+n_3\geq3$. Now, by taking $a=\pm(b+1)$, we see that $\Delta=a^2-4b=(b-1)^2$ is  a square number. On the other hand,  if $a=\pm(b+1)$, then $n_1=-a^2+(b+1)^2+1=1$ and $n_3=2a^2-(b+2)^2=b^2-2$. Clearly, $n_1+n_3\geq3$ whenever $b>2$.
\end{remark}
\begin{remark}\label{re-1}
It is clear that  $T_{2,0,0,n_4}$, $T_{n_1,0,n_3}$ and $T_{n_1,n_2}$ are infinite families since we can chose infinite  $a$ and $b$ to satisfy the restriction conditions such that $\Delta=a^2-4b$ is not a square number.

$T_{0,0,1,0,n_5}$ has restriction conditions $n_5=\frac{b^2-3}{2}\geq2$ and $2a^2=(b+2)^2+1$. By Lemma \ref{lem-pell}, $(x,y)=(b+2,a)$  is the solution of the negative Pell equation $x^2-2y^2=-1$ such that there are infinite  $P(b+2,a)=a^2-4((b+2)-2)=a^2-4b=\Delta$ are  square-free numbers. Thus $T_{0,0,1,0,n_5}$ is an infinite family, and so is $T_{0,0,0, n_4}$ since its restriction conditions are  same with $T_{0,0,1,0,n_5}$.
\end{remark}

At last of this paper, we use computer to search all the quadratic starlike trees $T_{0,0,1,0,n_5}$ with $n_5\leq1000$. It has been found that there are only five quadratic starlike trees of $T_{0,0,1,0,n_5}$-type within $5000$ vertices that are listed in Table \ref{table-T_{0,0,1,0,n_5}}. Also we see that the number of vertices of $T_{0,0,1,0,n_5}$ may be arbitrarily large but its number of eigenvalues remains to $13$.
\begin{table}
\scriptsize
\caption{\small $T_{0,0,1,0,n_5}(n_5\leq1000)$}
\centering
\renewcommand\arraystretch{1.3}
\begin{tabular*}{14cm}{m{4pt}|m{20pt}|m{20pt}|m{20pt}|m{20pt}|m{230pt}}
\hline
&$n_5$&$a$&$b$&$\Delta$&$f_{T_{0,0,1,0,n_5}}(x)$\\\hline
1&$3$&$1$&$-3$&$13$&$x^{3}(x^2-1)^{2}(x^2-x-1)(x^2+x-1)(x^2-3)^{2}(x^2-x-3)(x^2+x-3)$\\
2&$11$&$5$&$5$&$5$&$x^{11}(x^2-1)^{10}(x^2-x-1)(x^2+x-1)(x^2-3)^{10}(x^2-5x+5)(x^2+5x+5)$\\
 3&$39$&$5$&$-9$&$61$&$x^{39}(x^2-1)^{38}(x^2-x-1)(x^2+x-1)(x^2-3)^{38}(x^2-5x-9)(x^2+5x-9)$\\
4&$759$&$29$&$39$&$ 685$&$x^{759}(x^2-1)^{758}(x^2-x-1)(x^2+x-1)(x^2-3)^{758}(x^2-29x+39)(x^2+29x+39)$\\
5&$923$&$29$&$-43$&$1013$&$x^{923}(x^2-1)^{922}(x^2-x-1)(x^2+x-1)(x^2-3)^{922}(x^2-29x-43)(x^2+29x-43)$\\\hline
\end{tabular*}\label{table-T_{0,0,1,0,n_5}}
\end{table}


\begin{thebibliography}{11}{\small
\bibitem{Harary} Harary, F., Schwenk, A.J.: Which graphs have integral spectra?  in Graphs and Combinatorics. Lecture Notes in Math, vol. 406. Springer, Berlin (1974)\vspace{-0.3cm}
\bibitem{Patuzzi} Patuzzi, L., de Freitas, M.A.A., Del-Vecchio, R.R.: Indices for special classes of trees. Linear Algebra Appl. 442, 106--114 (2014) \vspace{-0.3cm}
\bibitem{Lu1} Lu, L., Huang, Q.X., Huang, X.Y.: Integral Cayley graphs over dihedral groups. J. Algebraic Combin. 47(4), 585--601 (2018) \vspace{-0.3cm}
\bibitem{Watanabe} Watanabe, M., Schwenk, A.J.: Integral starlike trees. J. Austral. Math. Soc.  28(1), 120--128 (1979)\vspace{-0.3cm}
\bibitem{Hic} H\'{i}c, P., Nedela, R.: Balanced integral trees. Math. Slovaca 48(5), 429--445 (1998)\vspace{-0.3cm}
\bibitem{Brouwer} Brouwer, A.E.: Small integral trees. Electron. J. Combin. 15(1), \#N1 (2008) \vspace{-0.3cm}
\bibitem{Omidi} Omidi, G.R.: On integral graphs with few cycles. Graphs Combin. 25(6), 841--849 (2009)\vspace{-0.3cm}
\bibitem{Csikvari} Csikv\'{a}ri, P.: Integral trees of arbitrarily large diameters. J. Algebraic Combin. 32(3), 371--377 (2010)\vspace{-0.3cm}
\bibitem{Zhang} Zhang, J., Huang, Q.X., Song, C.X., Huang, X.Y.: Q-integral unicyclic, bicyclic and tricyclic graphs. Math. Nachr. 290(5-6), 955--964 (2017) \vspace{-0.3cm}
\bibitem{Huang} Huang, X.Y., Huang, Q.X., Wen, F.: On the Laplacian integral tricyclic graphs. Linear Multilinear Algebra 63(7), 1356--1371 (2015)\vspace{-0.3cm}
\bibitem{Dragos} Cvetkovi\'{c}, D., Rowlinson, P., Simi\'{c}, S.: An Introduction to the Theory of Graph Spectra. Cambrige University Press, New York (2010)\vspace{-0.3cm}
\bibitem{Godsil} Godsil, C.D.: Spectra of trees. Ann. Discrete Math. 20, 151--159 (1984) \vspace{-0.3cm}
\bibitem{Mirko} Lepovi\'{c}, M., Gutman, I.: No starlike trees are cospectral. Discrete Math. 242(1-3), 291--295 (2002)\vspace{-0.3cm}
\bibitem{Brouwer2} Brouwer, A.E., Haemers, W.H.: Spectra of Graphs. Springer, New York (2012)\vspace{-0.3cm}
\bibitem{Lang} Lang, S.:  Algebra. Springer, New York (2002) \vspace{-0.3cm}
\bibitem{Washington} Washington, L.C.: Introduction to Cyclotomic Fields.  Springer, New York (1997)\vspace{-0.3cm}
\bibitem{Oboudi} Oboudi, M.R.: On the eigenvalues and spectral radius of starlike trees. Aequat. Math. 92(4), 683--694 (2018)\vspace{-0.3cm}
\bibitem{Koshy}Koshy, T.: Elementary Number Theory with Applications. Academic Press, New York (2007) \vspace{-0.3cm}
}
\end{thebibliography}
\end{document}